\RequirePackage{fix-cm}
\documentclass[smallextended]{svjour3}       
\usepackage{amsmath}
  \usepackage{paralist}
\usepackage{graphicx}

 \usepackage{amssymb}

\newcommand{\N}{\mathbb{N}}
\newcommand{\R}{\mathbb{R}}

 \journalname{journal}
\begin{document}

\title{\centerline{Layered solutions to the vector Allen-Cahn equation in $\R^2$.}\\
\centerline{Characterization of minimizers}\\\centerline{ and a new approach to heteroclinic connections.}}

\titlerunning{Layered solutions to the vector Allen-Cahn equation in $\R^2$}        

\author{\centerline{Giorgio Fusco}
}

\authorrunning{Fusco}

\institute{\at
              University of L'Aquila \\
            via Vetoio, Coppito\\
            67010 L'Aquila, ITALY \\
              \email{fusco@univaq.it}}           

\date{}

\maketitle

\begin{abstract}
Let $W:\R^m\rightarrow\R$ be a nonnegative potential with exactly two nondegenerate zeros $a_-\neq a_+\in\R^m$. We assume that there are $N\geq 1$ distinct heteroclinic orbits connecting $a_-$ to $a_+$ represented by maps  $\bar{u}_1,\ldots,\bar{u}_N$ that minimize the one-dimensional energy $J_\R(u)=\int_\R(\frac{\vert u^\prime\vert^2}{2}+W(u)){d} s$.

We first consider the problem of characterizing the minimizers $u:\R^n\rightarrow\R^m$ of the energy $\mathcal{J}_\Omega(u)=\int_\Omega(\frac{\vert\nabla u\vert^2}{2}+W(u)){d} x$. Under a nondegeneracy condition on $\bar{u}_j$, $j=1,\ldots,N$ and in two space dimensions, we prove that, provided it remains away from $a_-$ and $a_+$ in corresponding half spaces $S_-$ and $S_+$, a bounded minimizer $u:\R^2\rightarrow\R^m$  is necessarily an heteroclinic connection between suitable translates $\bar{u}_-(\cdot-\eta_-)$ and $\bar{u}_+(\cdot-\eta_+)$ of some $\bar{u}_\pm\in\{\bar{u}_1,\ldots,\bar{u}_N\}$.

Then we focus on the existence problem and
assuming $N=2$ and denoting $\bar{u}_-,\bar{u}_+$ the representations of the two orbits connecting $a_-$ to $a_+$  we give a new proof of the existence (first proved in \cite{scha}) of a solution $u:\R^2\rightarrow\R^m$ of
\[\Delta u=W_u(u),\]
that connects certain translates of $\bar{u}_\pm$.
\end{abstract}
\section{Introduction}

We consider entire solutions $u:\R^n\rightarrow\R^m$ of the vector Allen-Cahn equation
\begin{equation}\label{system}
\Delta u=W_u(u),
\end{equation}
where $W:\R^m\rightarrow\R$ is a $C^3$ potential and $W_u=(\frac{\partial W}{\partial u_1},\dots,\frac{\partial W}{\partial u_m})^\top$. We suppose that $W$ satisfies
\[\liminf_{\vert u\vert\rightarrow+\infty} W(u)>0,\]
and
\begin{eqnarray}
0=W(a_\pm)<W(u),\;\;u\not\in\{a_-,a_+\}
\end{eqnarray}
for some $a_-\neq a_+\in\R^m$ .

\noindent
We  assume that $a_\pm$ are non degenerate in the sense that the quadratic forms $W_{uu}(a_\pm)z\cdot z$ are positive definite.

These assumptions on $W$
ensure \cite{af2}, \cite{sourdis}, \cite{monteil}, \cite{ZS} the existence of an  heteroclinic connection between $a_-$ and $a_+$ represented by a map $\bar{u}:\R\rightarrow\R^m$ which is a minimizer of the problem
\begin{equation}\label{barmin}
\begin{split}
& J_\R(\bar{u})=\min_{\varphi\in\mathcal{A}}J_\R(\varphi),\;\;\;
J_\R(\varphi):=\int_\R(\frac{1}{2}\vert \varphi^\prime\vert^2+W(\varphi))ds,\\
&\mathcal{A}:=\{\varphi\in W_{\mathrm{loc}}^{1,2}(\R;\R^m):\lim_{s\to\pm\infty}\varphi(s)=a_\pm \}.
\end{split}
\end{equation}
We assume that there are exactly $N\geq 1$ distinct connections with representations $\bar{u}_1,\dots,\bar{u}_N$ that minimize $J_\R$ and define
\begin{equation}\label{c0}
c_0:= J_\R(\bar{u}_j),\;\;j=1,\ldots,N.
\end{equation}
We refer to \cite{abchen} for examples of potentials that allow multiple connections.
Given $\bar{u}\in\{\bar{u}_1,\dots,\bar{u}_N\}$, for each $r\in\R$
 the translate $\bar{u}(\cdot-r)$ is also a minimizer and therefore $0$ is in the spectrum of the linearized operator $T:W^{2,2}(\R;\R^m)\rightarrow L^2(\R;\R^m)$:
\begin{equation}\label{operator}
T\varphi=-\varphi^{\prime\prime}+W_{uu}(\bar{u})\varphi,\quad\bar{u}\in\{\bar{u}_1,\dots,\bar{u}_N\},
\end{equation}
and $\bar{u}^\prime$ is a corresponding eigenvector.
We assume that the $N$ connections are nondegenerate in the sense that $0$ is a simple eigenvalue of $T$ for $\bar{u}=\bar{u}_1,\dots,\bar{u}_N$.
\vskip.2cm
Is a standard fact that the assumption on $a_\pm$ being non degenerate and the smoothness of $W$ imply
 ($\sigma(s)=-$ for $s\leq 0$, $\sigma(s)=+$ for $s>0$)
\begin{equation}\label{n-bound}
\vert\bar{u}(s)-a_{\sigma(s)}\vert, \vert\bar{u}^\prime(s)\vert, \vert\bar{u}^{\prime\prime}(s)\vert\leq\bar{K}e^{-\bar{k}\vert s\vert},\;\;\text{for}\;\;s\in\R,
\;\,\bar{u}\in\{\bar{u}_1,\ldots,\bar{u}_N\},
\end{equation}
for some $\bar{k}, \bar{K}>0$.

For each map $\mathrm{u}\in W_{\rm{loc}}^{1,2}(\R^n;\R^m)\cap L^\infty(\R^n;\R^m)$ and for each open bounded and Lipschitz set $\Omega\subset\R^n$ we define
\[\mathcal{J}_\Omega(\mathrm{u}):=\int_\Omega(\frac{1}{2}\vert\nabla{\mathrm{u}}\vert^2+W(\mathrm{u}))dx.\]

\noindent
We focus on solutions of (\ref{system}) which are {\it minimizers}:
\begin{definition}\label{definition}
A map $u\in C^2(\mathcal{O};\R^m)\cap L^\infty(\mathcal{O};\R^m)$, $\mathcal{O}\subset\R^n$ open, is called a minimizer if
\[\mathcal{J}_\Omega(u)\leq\mathcal{J}_\Omega(\mathrm{u}),\]
for each open bounded and Lipschitz set $\Omega\subset\mathcal{O}$ and for each $\mathrm{u}\in W^{1,2}(\Omega;\R^m)\cap L^\infty(\Omega;\R^m)$ such that $\mathrm{u}=u$ on $\partial\Omega$.
\end{definition}
Classifying all minimizers $u=\R^n\rightarrow\R^m$ for general $n\geq 1$ is probably an impossible task. Even the scalar case $m=1$, in spite of many deep results \cite{gg}, \cite{ac}, \cite{aac}, \cite{gg1},  \cite{fsv}, \cite{s}, \cite{DKW} motivated by a famous conjecture of De Giorgi \cite{fa}, is far from being completely understood.
\vskip.2cm
We restrict to two space dimensions and give a complete characterization of all minimizers $u=\R^2\rightarrow\R^m$ that in half-spaces $S_-,S_+$ remain away from $a_-,a_+$ respectively. More precisely we prove
(recall that $\sigma(s)=-$ for $s\leq 0$, $\sigma(s)=+$ for $s>0$)
\begin{theorem}\label{main}
Assume $W$, $a_\pm$ and $\bar{u}_1,\dots,\bar{u}_N$ as before and let $u:\R^2\rightarrow\R^m$ be a minimizer that, for some $\delta>0$ and $\lambda>0$, satisfies the condition
\begin{equation}\label{avoid-a}
\vert u(x,y)-a_\pm\vert\geq\delta,\quad(x,y)\in S_\pm,\;\;
S_\pm=\{(x,y):\mp y\geq\lambda\}.
\end{equation}
Then there are $\bar{u}_\pm\in\{\bar{u}_1,\dots,\bar{u}_N\}$, numbers $\eta_\pm\in\R$ and constants $k, K$,  $k^\prime, K^\prime>0$ such that
\begin{equation}\label{asymptotic-u}
\begin{split}
&\vert u(x,y)-a_{\sigma(y)}\vert\leq Ke^{-k\vert y\vert},\\
&\vert u(x,y)-\bar{u}_{\sigma(x)}(y-\eta_{\sigma(x)})\vert\leq K^\prime e^{-k^\prime\vert x\vert}.
\end{split}
\end{equation}
Moreover if $\bar{u}_-=\bar{u}_+=\bar{u}$ it results
\begin{equation}\label{unique}
u(x,y)=\bar{u}(y-\eta),
\end{equation}
for some $\eta\in\R$.
\end{theorem}
Theorem \ref{main} says that under condition (\ref{avoid-a}) and in two dimensions a minimizer is necessarily an \emph{heteroclinic connection} between two minimizers of the \emph{effective potential}
\begin{equation}\label{eff-def}
\mathcal{W}(v):=J_\R(\bar{u}+v)-J_\R(\bar{u}),\quad c_0:= J_\R(\bar{u}),\;\;\bar{u}\in\{\bar{u}_1,\dots,\bar{u}_N\}.
\end{equation}
It is an open problem, which requires new ideas, to show that the same is true if condition (\ref{avoid-a}) is removed.

\noindent
We observe explicitly that if $N=1$ Theorem \ref{main} implies that a minimizer $u$ is necessarily a translate of $\bar{u}$. This is an example of a situation where a bounded smooth  solution $u:\R^n\rightarrow\R^m$ of (\ref{system}) is \emph{rigid} in the sense that there are: a subspace $V\subset\R^n$, $\mathrm{dim}V<n$, a projection $P:\R^n\rightarrow V$ and a map $\bar{u}:V\rightarrow\R^m$ such that
\[u(z)=\bar{u}(Pz),\;\text{ for }\;z\in\R^n.\]
In the case at hand $V=\{(x,y):x=0\}$ and $\bar{u}=\bar{u}_-(\cdot-\eta)=\bar{u}_+(\cdot-\eta)$.

Besides the already mentioned papers \cite{ac}, \cite{aac}, \cite{gg}, \cite{gg1},  \cite{fsv}, \cite{s}, \cite{DKW} concerning
the scalar case $m=1$, for the vector case $m>1$ we mention \cite{fg} where the authors, for $n=2$, and under a monotonicity assumption that involves $W$ and the components of $u$ prove that $u$ is one-dimensional. In \cite{fa1} for $m=2$, $n=3$ or $n=4$ and $W_u(u)=u(1-\vert u\vert^2)$ the author shows that a solution $u$ of (\ref{system}), which is a minimizer and satisfies $\vert u\vert\rightarrow 1$ as $\vert x\vert\rightarrow+\infty$, is necessarily constant. For rigidity results for the competitive system
\begin{equation}
\begin{split}
&\Delta u_j=\sum_{i\neq j}u_i^2u_j,\;\;j=1,\ldots,m,\\
& u_j>0,
\end{split}
\end{equation}
we refer to \cite{SoaveTerracini} and to the references therein.

Theorem \ref{main} can be considered as an infinite dimensional analogous of
\begin{theorem}\label{main-0}
Assume that $W\in C^3(\R^m;\R)$ satisfies
\[0=W(a)<W(z),\quad z\in\R^m\setminus\{a_1,\ldots,a_N\},\]
where $a_1,\ldots,a_N\in\R^m$ are distinct and non degenerate. Then, if $u:\R\rightarrow\R^m$ is a minimizer, either $u\equiv a$ for some $a\in\{a_1,\ldots,a_N\}$ or there exist $a_-\neq a_+\in\{a_1,\ldots,a_N\}$ such that
\[\vert u(x)-a_{\sigma(x)}\vert\leq Ke^{-k\vert x\vert},\quad x\in\R.\]
\end{theorem}
In this analogy the zeros $a_1,\ldots,a_N$ of $W$ correspond to the \emph{zeros} of the effective potential $\mathcal{W}$, that is, to the minimizers $\bar{u}_1,\dots,\bar{u}_N$ and heteroclinic connections between zeros of $W$ correspond to heteroclinic connections between zeros of $\mathcal{W}$. We note however that there is also an important difference between the zeros of $W$ that are isolated and the zeros of $\mathcal{W}$, that is the maps $\bar{u}_1,\ldots,\bar{u}_N$ and their translates which lie on $N$ distinct one-dimensional manifolds. Therefore in the case of Theorem \ref{main} there is the extra difficulty of proving that the heteroclinic is asymptotic to specific elements of the manifolds of the translates of $\bar{u}_-$ and $\bar{u}_+$.

In the proof of Theorem \ref{main} we regard the minimizer $u:\R^2\rightarrow\R^m$ as a map
 \[\R\ni x\rightarrow u(x,\cdot)\in\bar{u}+W^{1,2}(\R;\R^m)\]
 where $\bar{u}$ belongs to $\{\bar{u}_1,\ldots,\bar{u}_N\}$ or is any smooth map $\bar{u}:\R\rightarrow\R^m$ with the same asymptotic behaviour. We interpret the elliptic system (\ref{system}) as an ODE in the infinite dimensional function space
 $\bar{u}+W^{1,2}(\R;\R^m)$:
  \[u_{xx}(x,\cdot)=\nabla_{L^2(\R;\R^m)}(J_\R(u(x,\cdot))-c_0).\]
 In the above analogy this equation corresponds to the equation $u^{\prime\prime}=W_u(u)$ satisfied by $\bar{u}_1,\ldots,\bar{u}_N$. This point of view and  a particular representation formula for the energy based on two \emph{hamiltonian identities} that are consequences of the minimality of $u$ allow to adapt to the case at hand some of the ideas in \cite{af2}.

 After completing the proof of Theorem \ref{main} we restrict to the case $N=2$ and turn to the problem of the \emph{existence} of solutions of (\ref{system}) that satisfies (\ref{asymptotic-u}). This question was first considered in \cite{abg} for potential invariant under the reflection that exchanges $a_-$ and $a_+$ and the existence of a symmetric solution that connects symmetric minimizers $\bar{u}_-$ and $\bar{u}_+$ of $J_\R$ was established. Working in the symmetric context greatly simplify the problem and, in particular, fixes the position of the \emph{interface} (which can be identified with the $x$ axis) that separates the regions where $u$ is near $a_-$ or near $a_+$ and automatically implies $\eta_\pm=0$ which otherwise are unknowns of the problem.

 \noindent
 The question was solved in full generality in a remarkable paper \cite{scha} by an approach based on variational arguments and on dynamical systems techniques used for the construction of solutions of (\ref{system}) in half spaces with Dirichlet conditions near $\bar{u}_-$ or near $\bar{u}_+$. The result proved in \cite{scha} is the following
 \begin{theorem}\label{scha-th}
 Let $W\in C^3(\R^m;\R)$, $a_\pm$ be as before and assume $N=2$. Assume that
\begin{equation}\label{M}
W(sz)\geq W(z),\quad\text{ for }\;\vert z\vert\geq M,\;s\geq 1.
\end{equation}
for some $M>0$.
Then there exists  $u\in C^{2+\alpha}(\R^2;\R^m)$ and $\eta_\pm\in\R$ that solve (\ref{system}) and satisfy (\ref{asymptotic-u}).
\end{theorem}
Our aim is to show that some of the arguments in the proof of Theorem \ref{main} can be employed to construct a new, nontrivial but elementary, proof of Theorem \ref{scha-th}. As in the proof of Theorem \ref{main} we regard the sought solution of (\ref{system}) as a map $\R\ni x\rightarrow u(x,\cdot)\in \bar{u}+W^{1,2}(\R;\R^m)$, $\bar{u}\in\{\bar{u}_-,\bar{u}_+\}$ that connects the minimizers $\bar{u}_\pm(\cdot-\eta_\pm)$ of the effective potential $\mathcal{W}$. In other words we look at the problem as the problem of the existence of an heteroclinic orbit in the function space $\bar{u}+W^{1,2}(\R;\R^m)$.

Beside the already mentioned papers \cite{abg} and \cite{scha} there are other works dealing with the existence of heteroclinic solutions to PDEs. Solutions homoclinic to zero were considered in \cite{al}. Heteroclinic solutions in periodic domains are the object of \cite{br} and \cite{af3}. In \cite{R}, in a periodic setting, existence of solutions connecting two different periodic functions was established.
\vskip.2cm
We now describe our approach to the proof of Theorem \ref{scha-th} and discuss how we can overcome certain difficulties that prevent from a straightforward application of the direct method of variational calculous. These are: loss of compactness due to translation invariance in the $x$ and $y$ directions and the fact that the solution we are looking for has infinite energy. To deal with these obstructions we consider a bounded strip $\mathcal{R}_L=\{(x,y):x\in(0,L); y\in\R\}$, $L>1$  and, for each $\eta\in\R$, consider the problem
\begin{equation}\label{P}
\begin{split}
&\min_{\mathcal{A}_{L,\eta}}\mathcal{J}({\rm{u}}),\quad \mathcal{J}({\rm{u}})=\int_{\mathcal{R}_L}\Big(W({\rm{u}})+\frac{1}{2}\vert\nabla {\rm{u}}\vert^2\Big)dxdy,\\
&\mathcal{A}_{L,\eta}=\{{\rm{u}}\in W_{\mathrm{loc}}^{1,2}(\mathcal{R}_L;\R^m):{\rm{u}}(\cdot,0)=\bar{u}_-,\;{\rm{u}}(\cdot,L)=\bar{u}_+(\cdot-\eta)\}.
\end{split}
\end{equation}
Working in a bounded strip with imposed Dirichlet conditions removes at once the difficulties mentioned before. But on the other hand raises the problem of understanding the relationship between minimizers $u^{L,\eta}$ of (\ref{P}) and the solution $u$ we are looking for. We regard the minimization problem (\ref{P}) as a first step where we impose to the minimizer to connect two fixed elements of the manifolds of the translates of $\bar{u}_\pm$. We note in passing that what actually matter is the difference $\eta=\eta_+-\eta_-$ rather than the values of $\eta_-$ and $\eta_+$ separately. Indeed a translation in the $y$ direction reduces to the case considered in (\ref{P}).

By mean of the Cut-Off Lemma (see section 2.2 in \cite{af3}) we show that the admissible set $\mathcal{A}_{L,\eta}$ in (\ref{P}) can be restricted to maps that converge to $a_\pm$ as $y\rightarrow\pm\infty$ with a well controlled rate. Then
standard arguments imply that, given $L>1$, there exist a minimizer $u^{L,\eta}\in\mathcal{A}_{L,\eta}$ of problem (\ref{P}) for each $\eta\in\R$ and a map $u^L$ that satisfies the condition
\begin{equation}\label{condition}
\mathcal{J}(u^L)=\min_{\eta\in\R}\mathcal{J}(u^{L,\eta}),
\end{equation}
that we use to determine the value of $\eta$.
This yields a family of maps $u^L,\;L>1$ which, for large $L$, are expected to be good approximations of a translation of the sought solution $u$ in Theorem \ref{scha-th}. Therefore we expect that
\begin{equation}\label{uS-lim}
u(x,y)=\lim_{L_j\rightarrow+\infty}u^{L_j}(x-l_j,y),
\end{equation}
for suitable sequences $L_j,\;l_j,\;j=1,\ldots$ To show that this is indeed the case we need to derive precise point-wise
estimates on $u^L$.

The paper is organized as follows. We prove Theorem \ref{main} In Section \ref{rigidity} and Theorem \ref{scha-th}  in Section \ref{schatzman}. In Section \ref{section21} we study the projection on the manifold of the translates of $\bar{u}\in\{\bar{u}_1,\ldots,\bar{u}_N\}$. In Section \ref{section22} we study the effective potential and show that, away from $\bar{u}_1,\ldots,\bar{u}_N$ and their translates, $\mathcal{W}$ is bounded below by a constant larger than $c_0$. In Section \ref{section23} we discuss two \emph{hamiltonian identities} that are used to derive a particular representation formula for the energy which is essential for obtaining the point-wise estimates needed for completing the proof of Theorem \ref{main} in Section \ref{section24}.
In Section \ref{DUE} we sketch the proof of Theorem \ref{main-0}.
In Section \ref{section31} we prove existence of $u^{L,\eta}$ and $u^L$. In Section \ref{section32} we derive information on the asymptotic behavior of $u^{L,\eta}$ and $u^L$. In Section \ref{section33} we prove that $u^L\in\mathcal{A}_{L,\bar{\eta}}$ for some $\bar{\eta}$ which is bounded independently of $L>1$. This is a key point for the analysis that concludes the proof of Theorem \ref{scha-th} in Section \ref{section34}.

In the following we write $z\in\R^2$ as $z=(x,y)$ with $x,y\in\R$; we denote by $z\cdot\zeta$ the scalar product between $z,\zeta\in\R^d$, $d\geq 1$, by $\langle u,v\rangle$ the standard inner product in $L^2(\R;\R^m)$ and by $\|u\|=\sqrt{\langle u,u\rangle}$ the associated norm. For a measurable set $S\subset\R^d$, $\vert S\vert$ denotes the $d$-dimensional Lebesgue measure of $S$. We denote by $C, k, K$ generic positive constants that may depend on $W, M, \delta, \lambda$. The value of $C, k, K$ may change from line to line.

\section{The proof of Theorem \ref{main}}\label{rigidity}

Since $u$ is, by assumption, a bounded solution of (\ref{system}) and $W\in C^3(\R^m;\R)$ elliptic theory implies that
\begin{equation}\label{C2-bound}
\|u\|_{C^{2,\gamma}(\R^2;\R^m)}\leq M,
\end{equation}
for some $M>0, \gamma\in(0,1)$.
\begin{lemma}\label{point-exp}
There exist $k, K>0$ such that, for $\alpha\in\N^2$, $\alpha_1+\alpha_2\leq 2$, it results
\begin{equation}\label{exp-der-decay}
\vert (D^\alpha(u-\bar{u}))(x,y)\vert\leq Ke^{-k\vert y\vert},\;\;\alpha=(\alpha_1,\alpha_2),\;\;\bar{u}=\bar{u}_1,\ldots,\bar{u}_N,
\end{equation}
\begin{equation}\label{l2-der-norm}
\|(D^\alpha(u-\bar{u}))(x,\cdot)\|\leq\frac{K}{\sqrt{k}},\;\; x\in\R,\;\bar{u}=\bar{u}_1,\ldots,\bar{u}_N.
\end{equation}
\end{lemma}
\begin{proof}
From (\ref{avoid-a}) we see that the minimizer $u$ satisfies  the hypothesis of Theorem 1.2 in \cite{f2} or Theorem 3.1 in \cite{af4} with respect to $a_-$ in the half space $S_-=\{(x,y):y\geq\lambda\}$ and with respect to $a_+$ in $S_+=\{(x,y):y\leq-\lambda\}$. It follows
\begin{equation}\label{FirstSta}
\vert u(x,y)-a_{\sigma(y)}\vert\leq K e^{-k\vert y\vert},
\end{equation}
for some constants $k, K>0$. Therefore from
 (\ref{n-bound}) we obtain
\begin{equation}\label{v-estimate}
\vert u(x,y)-\bar{u}(y)\vert\leq K e^{-k\vert y\vert},\;\;(x,y)\in\R^2,\;\bar{u}=\bar{u}_1,\ldots,\bar{u}_N.
\end{equation}
This estimate and elliptic interior regularity imply (\ref{exp-der-decay}).
The bound (\ref{l2-der-norm}) is a plain consequence of (\ref{exp-der-decay}). The proof is complete.
\end{proof}
A simple and useful consequence of the estimate (\ref{exp-der-decay}) is that in the Definition \ref{definition} of minimality of $u$ we can extend the class of sets $\Omega$ to include strips aligned with the $y$ axis.
\begin{lemma}\label{extend}
Let $u$ the minimizer in Lemma \ref{point-exp}. Given $x_0\in\R$ and $L>0$ let $\mathcal{R}_L(x_0):=(x_0,x_0+L)\times\R$. Then
\[\mathcal{J}_{\mathcal{R}_L(x_0)}(u)\leq\mathcal{J}_{\mathcal{R}_L(x_0)}(v),\]
for all $v\in u+W^{1,2}(\R^2;\R^m)$ that satisfy
\[v(x_0,y)=u(x_0,y),\;\;v(x_0+L,y)=u(x_0+L,y),\]
and
\[\vert v(x,y)-a_{\sigma(y)}\vert,\,\vert(\nabla u)(x,y)\vert\leq K e^{-k\vert y\vert}.\]
\end{lemma}
\begin{proof}
The proof is elementary. See Lemma 6.2 in \cite{af4}.
\end{proof}
The minimality of $u$ imply an upper bound for the energy
\begin{lemma}\label{upper-bound}
There exists $C_0>0$ independent of $x_0\in\R$ and $L>0$ such that
\begin{equation}\label{upper-bound-1}
\mathcal{J}_{\mathcal{R}_L(x_0)}(u)\leq c_0L+C_0,
\end{equation}
where $c_0$ is defined in (\ref{c0}).
\end{lemma}
\begin{proof}
Fix $\bar{u}\in\{\bar{u}_1,\ldots,\bar{u}_N\}$, assume $L>2$ and define a competing map $v$ by setting
\begin{equation}\label{comp-map}
v(x,\cdot)=\left\{\begin{array}{l}
u(x,\cdot),\;\;x\in(-\infty,x_0]\cup[x_0+L,+\infty),\\
(1-(x-x_0))u(x_0,\cdot)+(x-x_0)\bar{u}(\cdot),\;\;x\in(x_0,x_0+1),\\
\bar{u},\;\;x\in[x_0+1,x_0+L-1],\\
(x_0+L-x)\bar{u}(\cdot)+(1+x-x_0-L)u(x_0+L,\cdot),\\\quad x\in(x_0+L-1,x_0+L).
\end{array}\right.
\end{equation}
From this definition, (\ref{n-bound}) and Lemma \ref{point-exp} it follows that, for $x\in[x_0,x_0+1]\cup[x_0+L-1,x_0+L]$, $v$ satisfies
\begin{equation}\label{v-der}
\vert v_x(x,y)\vert^2, \vert v_y(x,y)\vert^2\leq K^2 e^{-2k\vert y\vert}.
\end{equation}
On the other hand $v$ is bounded and converges exponentially to $a_\pm$ as $y\rightarrow\pm\infty$ therefore we have
\[\lim_{y\rightarrow\pm\infty}W(v(x,\cdot))=0,\;\;x\in[x_0,x_0+1]\cup[x_0+L-1,x_0+L].\]
with exponential convergence. This and (\ref{v-der}) imply
\[\mathcal{J}_{[x_0,x_0+1]\cup[x_0+L-1,x_0+L]\times\R}(v)\leq C_0\]
for some constant $C_0>0$
and therefore
\[\mathcal{J}_{\mathcal{R}_L(x_0)}(v)\leq c_0L+C_0.\]
To conclude the proof we note that (\ref{comp-map}) implies $v(x_0,\cdot)=u(x_0,\cdot)$ and $v(x_0+L,\cdot)=u(x_0+L,\cdot)$. Then  Lemma \ref{extend} implies
\[\mathcal{J}_{\mathcal{R}_L(x_0)}(u)\leq \mathcal{J}_{\mathcal{R}_L(x_0)}(v)\leq c_0L+C_0.\]
The proof is complete.
\end{proof}

Lemma \ref{upper-bound} implies an upper bound for the kinetic energy.
\begin{lemma}\label{kineticUpperBound}
It results
\begin{equation}\label{KUpperBound}
\int_\R\int_\R\vert u_x\vert^2{d} y{d} x\leq C_0.
\end{equation}
\begin{equation}\label{KineticLim}
\lim_{x\rightarrow\pm\infty}\int_\R\vert u_x(x,y)\vert^2{d} y=0.
\end{equation}
\end{lemma}
\begin{proof}
The minimality of $\bar{u}_1,\ldots,\bar{u}_N$ implies
\[J_\R(u(x,\cdot))-c_0\geq 0,\;\;x\in[x_0,x_0+L]\]
and therefore from (\ref{upper-bound-1}) we obtain
\[\int_{x_0}^{x_0+L}\int_\R\vert u_x\vert^2{d} y{d} x\leq\int_{x_0}^{x_0+L}(J_\R(u(x,\cdot))-c_0){d} x+\int_{x_0}^{x_0+L}\int_\R\vert u_x\vert^2{d} y{d} x\leq C_0\]
and, since this inequality is valid for all $x_0\in\R$ and $L>0$, (\ref{KUpperBound}) follows.
If (\ref{KineticLim}) does not hold we have
\[\int_\R\vert u_x(x_j,y)\vert^2{d} y\geq\epsilon>0\]
along a sequence $x_j,\; j=1,\ldots$ that converges either to $-\infty$ or to $+\infty$. Lemma \ref{point-exp} implies
\[\vert\frac{d}{d x}\int_\R\vert u_x(x,y)\vert^2{d} y\vert\leq2\|u_x(x,\cdot)\|u_{xx}(x,\cdot)\|\leq\frac{2K^2}{k}.\]
It follows that the map $x\rightarrow\int_\R\vert u_x(x,y)\vert^2{d} y$ is Lipschitz continuous and we have
\[\int_\R\vert u_x(x,y)\vert^2{d} y\geq\frac{\epsilon}{2},\;\;x\in(x_j-\delta,x_j+\delta)\]
for some $\delta>0$ independent of $j=1,\ldots$. By passing to a subsequence we can assume $\vert x_{j+1}-x_j\vert\geq2\delta$ and conclude (assuming for example that $x_j\rightarrow+\infty$)
\[j\delta\epsilon\leq\int_{x_1-\delta}^{x_j+\delta}\int_\R\vert u_x(x,y)\vert^2{d} y\leq C_0\]
that contradicts (\ref{KUpperBound}) for large $j$. This concludes the proof.
\end{proof}

\subsection{The decomposition of a map ${\rm{u}}$ near a translate of $\bar{u}_j$.}\label{section21}
Recall that $\|\cdot\|$ denote the norm in $L^2(\R;\R^m)$. We have (see also Lemma 2.1 in \cite{scha})
\begin{lemma}\label{lemmaw}
Let ${\rm{u}}:\R\rightarrow\R^m$ be a map that satisfy
\begin{equation}\label{inW12}
{\rm{u}}\in\bar{u}_1+ L^2(\R,\R^m).\footnote{The space $\bar{u}_1+ L^2(\R,\R^m)$ does not change if $\bar{u}_1$ is replaced by any smooth map with the same asymptotic behavior.}
\end{equation}
Then
\begin{enumerate}
\item  there exist ${h}\in\R$  and $\bar{u}\in\{\bar{u}_1,\ldots,\bar{u}_N\}$ such that
\begin{equation}\label{etabar}
\|{\rm{u}}-\bar{u}(\cdot-{h})\|=\min_j\min_{r}\|{\rm{u}}-\bar{u}_j(\cdot-{r})\|.
\end{equation}
Moreover it results
\begin{equation}\label{orthogonal}
\langle {\rm{u}}-\bar{u}(\cdot-{h}),\bar{u}^\prime(\cdot-{h})\rangle=0.
\end{equation}
\item There exists $q^0>0$ such that, for $q:=\|{\rm{u}}-\bar{u}(\cdot-{h})\|\leq q^0$,  ${h}$ and $\bar{u}$ are uniquely determined.
\item  For $q<q^0$, ${h}$ is a $C^2$ function of ${\rm{u}}$ and it results

\begin{equation}\label{derivative}
\begin{split}
&(D_{\rm{u}}h)w=-\frac{\langle w,\bar{u}^\prime(\cdot-h)\rangle}{\|\bar{u}^\prime\|^2-\langle {\rm{u}}-\bar{u}(\cdot-h),\bar{u}^{\prime\prime}(\cdot-h)\rangle}.
\end{split}
\end{equation}
\end{enumerate}
\end{lemma}
\begin{proof}

For large $\vert{r}_1-{r}_2\vert$ we have with $a=\frac{a_+-a_-}{2}$
\[\|\bar{u}_j(\cdot-{r}_1)-\bar{u}_j(\cdot-{r}_2)\|^2\geq\vert a\vert^2\vert{r}_1-{r}_2\vert,\;\;j=1,\ldots,N,\]
and therefore
\begin{equation}\label{compact}
\|{\rm{u}}-\bar{u}_j(\cdot-{r})\|\geq\vert a\vert\sqrt{\vert{r}\vert}-\|{\rm{u}}-\bar{u}_j\|,\;\;j=1,\ldots,N.
\end{equation}
This proves 1. The fact that (\ref{orthogonal}) holds is standard.

To prove 2. we observe that the assumption that the $N$ connections are distinct implies that the $L^2$ distance between distinct $\bar{u}_i, \bar{u}_j$ and their translates has a positive lower bound
\begin{equation}\label{q*}
q^*=\min_{i\neq j}\min_{r\in\R}\|\bar{u}_i(\cdot)-\bar{u}_j(\cdot-{r})\|>0.
\end{equation}
It follows that, for $q^0\in(0,q^*)$, $\bar{u}$ in (\ref{etabar}) is uniquely determined.

To show that, provided $q^0\in(0,q^*)$ is sufficiently small, for $q\leq q^0$, also ${h}$ is uniquely determined  we observe that from (\ref{compact}) it suffices to show that $\bar{u}(\cdot-{h})$ is the only choice that satisfies (\ref{etabar}) for $\vert{r}\vert\leq{h}_0$ for some large ${h}_0>0$.

For $q\leq q^0$ we have
\[\|{\rm{u}}-\bar{u}(\cdot-{r})\|\leq q+\|\bar{u}(\cdot-{r})-\bar{u}(\cdot-{h})\|
\leq q^0+\|\bar{u}^\prime\|\vert{r}-{h}\vert.\]

It follows
\begin{equation}\label{mi-serve}
\|\bar{u}^\prime\|^2-\langle\bar{u}^{\prime\prime}(\cdot-{r}),{\rm{u}}-\bar{u}(\cdot-{r})\rangle
\geq\frac{\|\bar{u}^\prime\|^2}{2},
\end{equation}
provided $\vert{r}-{h}\vert\leq\frac{\|\bar{u}^\prime\|}{4\|\bar{u}^{\prime\prime}\|}$ and $q\leq q^0\leq
\frac{\|\bar{u}^\prime\|^2}{4\|\bar{u}^{\prime\prime}\|}$.
This and
\begin{equation}\label{v-baru-der}
\begin{split}
&\frac{d}{d{r}}\|{\rm{u}}-\bar{u}(\cdot-{r})\|^2=2\langle\bar{u}^\prime(\cdot-{r}),{\rm{u}}-\bar{u}(\cdot-{r})\rangle\\
&=2\Big(\langle\bar{u}^\prime(\cdot-{r}),{\rm{u}}-\bar{u}(\cdot-{r})\rangle
-\langle\bar{u}^\prime(\cdot-{h}),{\rm{u}}-\bar{u}(\cdot-{h})\rangle\Big)\\
&=2\int_{{h}}^{r}\frac{d}{ds}\langle\bar{u}^\prime(\cdot-{s}),{\rm{u}}-\bar{u}(\cdot-{s})\rangle ds\\
&=2\int_{{h}}^{r}(\|\bar{u}^\prime\|^2-\langle\bar{u}^{\prime\prime}(\cdot-{s}),{\rm{u}}-\bar{u}(\cdot-{s})\rangle) ds,
\end{split}
\end{equation}
yield, for $\vert r-h\vert\leq\frac{\|\bar{u}^\prime\|}{4\|\bar{u}^{\prime\prime}\|}$,
\begin{equation}\label{v-baru-der1}
\begin{split}
&\frac{d}{d{r}}\|{\rm{u}}-\bar{u}(\cdot-{r})\|^2\geq\|\bar{u}^\prime\|^2({r}-{h}),\quad\text{ for }\;\;{r}\geq{h}\\
&\frac{d}{d{r}}\|{\rm{u}}-\bar{u}(\cdot-{r})\|^2\leq\|\bar{u}^\prime\|^2({r}-{h}),\quad\text{ for }\;\;{r}\leq{h}
\end{split}
\end{equation}
and we obtain, for $\vert{r}-{h}\vert\leq\frac{\|\bar{u}^\prime\|}{4\|\bar{u}^{\prime\prime}\|}$ and $q\leq q^0\leq
\frac{\|\bar{u}^\prime\|^2}{4\|\bar{u}^{\prime\prime}\|}$,
\begin{equation}\label{q-qeta}
\|{\rm{u}}-\bar{u}(\cdot-{r})\|^2\geq q^2+\frac{1}{2}\|\bar{u}^\prime\|^2\vert{r}-{h}\vert^2.
\end{equation}
This shows that ${h}$ is the unique solution of (\ref{etabar}) in the interval $\vert{r}-{h}\vert\leq\frac{\|\bar{u}^\prime\|}{4\|\bar{u}^{\prime\prime}\|}$. We have already observed that no choice of $\vert{r}-{h}\vert\geq{h}_0$ satisfies (\ref{etabar}). To conclude the proof of uniqueness it remains to prove that there are no solutions for $\frac{\|\bar{u}^\prime\|}{4\|\bar{u}^{\prime\prime}\|}\leq\vert{r}-{h}\vert\leq{h}_0$. Define
\[q_m:=\min_{\frac{\|\bar{u}^\prime\|}{4\|\bar{u}^{\prime\prime}\|}\leq\vert{r}-{h}\vert\leq{h}_0}
\|\bar{u}(\cdot-{r})-\bar{u}(\cdot-{h})\|\]
then $q_m>0$ and, by reducing $q^0$ if necessary, we can assume that $q_m>2q^0$ then, for $\frac{\|\bar{u}^\prime\|}{4\|\bar{u}^{\prime\prime}\|}\leq\vert{r}-{h}\vert\leq{h}_0$, we have
\[\|{\rm{u}}-\bar{u}(\cdot-{r})\|\geq\|\bar{u}(\cdot-{r})-\bar{u}(\cdot-{h})\|-q\geq q_m-q^0>q^0.\]
This concludes the proof of 2.

Note that (\ref{v-baru-der}) and (\ref{v-baru-der1}) imply that, for $q\leq q^0$, ${h}$ is the unique solution of (\ref{orthogonal}) in the interval $\vert{r}-{h}\vert\leq\frac{\|\bar{u}^\prime\|}{4\|\bar{u}^{\prime\prime}\|}$. Therefore, under these condition, ${h}$ solves (\ref{etabar}) if and only if solves (\ref{orthogonal}). It follows that we can analyze the dependence of ${h}$ on ${\rm{u}}$ by looking at the equation

\[f(h,{\rm{u}}):=\langle {\rm{u}}-\bar{u}(\cdot-{h}),\bar{u}^\prime(\cdot-{h})\rangle=0,\quad\text{ for }\;\;q<q^0.\]  The assumption that $W$ is $C^3$ implies that $f$ is of class $C^2$. Moreover $D_hf(h,{\rm{u}})$ coincides with the left hand side of (\ref{mi-serve}) and does not vanish for $q<q^0$ therefore the implicit function theorem implies that the solution of (\ref{orthogonal}) is a $C^2$ function of ${\rm{u}}$ and (\ref{derivative}) follows by differentiating (\ref{orthogonal}). The proof is complete.
 \end{proof}

 \begin{remark}\label{remark}
 If in (\ref{etabar}) we replace the $L^2$ norm $\|\cdot\|$ with the $W^{1,2}$ norm $\|\cdot\|_1$, then by repeating verbatim the steps in the proof of Lemma \ref{lemmaw}, we establish the analogous of Lemma \ref{lemmaw} for the $W^{1,2}$ norm. In particular we have that, for $q^0>0$ sufficiently small, the condition
 \begin{equation}\label{q01}
 \min_j\min_{r\in\R}\|\mathrm{u}-\bar{u}_j(\cdot-{r})\|_1\leq p,\;\;p\in(0,q^0]
 \end{equation}
 implies the existence of unique ${h}_1$ and $\bar{u}$ that solves (\ref{etabar}) with $\|\cdot\|_1$ instead of $\|\cdot\|$ and $\bar{u}$ does not depend on which norm is used. As expected the difference ${h}-{h}_1$ between the solutions ${h}$ and ${h}_1$ of (\ref{etabar}) in the $L^2$ and $W^{1,2}$ sense converges to zero with $p$.
 \end{remark}
 \begin{lemma}\label{lemmaw1}
 If (\ref{q01}) holds with $p\in(0,q^0]$ and $q^0>0$ is sufficiently small, then the solutions ${h}$ and ${h}_1$ of (\ref{etabar}) in the $L^2$ and the $W^{1,2}$ sense respectively satisfy
 \begin{equation}\label{h-h}
 \vert{h}-{h}_1\vert\leq\frac{2^\frac{1}{2}}{\|\bar{u}^\prime\|}{p}^\frac{1}{2}.
 \end{equation}
 Moreover it results
 \begin{equation}\label{w12-diffe}
 \|\mathrm{u}-\bar{u}(\cdot-{h})\|_1\leq22^\frac{1}{2}\frac{\|\bar{u}^\prime\|_1}{\|\bar{u}^\prime\|} (p)^\frac{1}{2}=\bar{C}{p}^\frac{1}{2}.
 \end{equation}
 \end{lemma}
 \begin{proof}
 From (\ref{q-qeta}) we have
 \[\begin{split}
 &(q^0)^2\geq p^2=\|\mathrm{u}-\bar{u}(\cdot-{h}_1)\|_1^2\geq\|\mathrm{u}-\bar{u}(\cdot-{h}_1)\|^2\\
 &\geq\|\mathrm{u}-\bar{u}(\cdot-{h})\|^2+\frac{1}{2}\|\bar{u}^\prime\|^2\vert{h}-{h}_1\vert^2
 \end{split}\]
 and (\ref{h-h}) is established.
 From (\ref{h-h}) it follows
 \[\|\bar{u}(\cdot-{h})-\bar{u}(\cdot-{h}_1)\|_1\leq\|\bar{u}^\prime\|_1\vert{h}-{h}_1\vert
 \leq2^\frac{1}{2}\frac{\|\bar{u}^\prime\|_1}{\|\bar{u}^\prime\|}p^\frac{1}{2}\]
 and therefore
 \[\begin{split}
 &\|\mathrm{u}-\bar{u}(\cdot-{h})\|_1\leq\|\mathrm{u}-\bar{u}(\cdot-{h}_1)\|_1
 +\|\bar{u}(\cdot-{h})-\bar{u}(\cdot-{h}_1)\|_1\\
 &\leq p+2^\frac{1}{2}\frac{\|\bar{u}^\prime\|_1}{\|\bar{u}^\prime\|}p^\frac{1}{2}
 \leq22^\frac{1}{2}\frac{\|\bar{u}^\prime\|_1}{\|\bar{u}^\prime\|}p^\frac{1}{2},
 \quad\text{ for }\;\;p\leq q^0\leq2(\frac{\|\bar{u}^\prime\|_1}{\|\bar{u}^\prime\|})^2.
 \end{split}\]
 The proof is complete.
 \end{proof}

 \begin{lemma}\label{hprime-hpprime?}
 Let $I\subset\R$ be an open interval and let $u$ be a minimizer that satisfies
 \begin{equation}\label{exp-der-decay?}
\vert (D^\alpha(u-\bar{u}_1))(x,y)\vert\leq Ke^{-k\vert y\vert},\;\; (x,y)\in I\times\R,\;\alpha\in\N^2,\;\alpha_1+\alpha_2\leq 2,
\end{equation}
for some $k, K>0$.
Assume that
\[q(x)=\min_j\min_{r\in\R}\|u(x,\cdot)-\bar{u}_j(\cdot-{r})\|\leq q^0,\quad x\in I,\]
with $q^0>0$ sufficiently small.
\vskip.2cm

 Then, for each $x\in I$, there are unique $h(x)\in\R$ and $\bar{u}\in\{\bar{u}_1,\ldots,\bar{u}_N\}$ that satisfy
 \begin{equation}\label{qx}
 q(x)=\|u(x,\cdot)-\bar{u}(\cdot-h(x))\|,
 \end{equation}
 and $\bar{u}$ is the same for all $x\in I$.

 \noindent
  Moreover, for each $x\in I$, $u(x,\cdot)$ can be uniquely decomposed in the form
 \begin{equation}\label{decomp-i1}
 \begin{split}
 &u(x,\cdot)=\bar{u}(\cdot-h(x))+{v}(x,\cdot-h(x))\\
 &{v}(x,y):= u(x,y+h(x))-\bar{u}(y),\\
 &\langle{v}(x,\cdot),\bar{u}^\prime\rangle=0,
 \end{split}
 \end{equation}
 and, if $q(x)>0$, it results $v(x,\cdot)=q(x)\nu(x,\cdot)$
 with $\|\nu(x,\cdot)\|=1$.

 The function $h:I\rightarrow\R$ is of class $C^2$ and
 \[\vert h\vert, \vert h^\prime\vert\leq C,\]
 for some constant $C>0$.

 Finally there exists $k_1, K_1>0$ such that
 \begin{equation}\label{exp-der-decay1?}
\vert (D_y^i{v})(x,y)\vert\leq K_1e^{-k_1\vert y\vert},\quad\text{ for }\;\; (x,y)\in I\times\R,\;i=0,1,2
\end{equation}
and
\begin{equation}\label{l2-der-norm1?}
\|(D_y^i{v})(x,\cdot)\|\leq\frac{K_1}{\sqrt{k_1}},\quad\text{ for }\;\; x\in I,\;i=0,1,2.
\end{equation}
 \end{lemma}
\begin{proof}
Existence and uniqueness of $\bar{u}$ and of the map $h:I\rightarrow\R$ that satisfies (\ref{qx}) and (\ref{decomp-i1}) follow from Lemma \ref{lemmaw}. From (\ref{exp-der-decay?}) it follows that $h$ is bounded. To show that the same is true for $h^\prime$ we invoke  3. in Lemma \ref{lemmaw} which implies
\[h^\prime(x)=-\frac{\langle u_x(x,\cdot),\bar{u}^\prime(\cdot-h(x))\rangle}{\|\bar{u}^\prime\|^2-\langle {u(x,\cdot)-\bar{u}(\cdot-h(x)),\bar{u}^{\prime\prime}(\cdot-h(x))\rangle}},\quad x\in I\]
and therefore, for $q^0\leq\frac{\|\bar{u}^\prime\|^2}{2\|\bar{u}^{\prime\prime}\|}$, we have
\[\vert h^\prime(x)\vert\leq\frac{\|u_x(x,\cdot)\|\|\bar{u}^\prime\|}{\|\bar{u}^\prime\|^2-q(x)\|\bar{u}^{\prime\prime}\|}
\leq 2\frac{K_1}{\sqrt{k_1}\|\bar{u}^\prime\|},\quad x\in I\]
where we have used (\ref{exp-der-decay?}) and $(u-\bar{u}_1)_x=u_x$ that imply
\begin{equation}\label{l2-der-norm?}
\|u_x(x,\cdot)\|\leq\frac{K_1}{\sqrt{k_1}},\quad\text{ for }\;\; x\in I.
\end{equation}
The estimate (\ref{exp-der-decay1?}) follows from (\ref{exp-der-decay?}), from the bound on $h$ and (\ref{n-bound}). The estimate (\ref{l2-der-norm1?}) from (\ref{exp-der-decay1?}). The proof is complete.
\end{proof}

\subsection{The effective potential.}\label{section22}
\begin{lemma}\label{L2mpliesLinfty}
Let $v\in L^2(\R;\R^m)$ be a $C^1$ map such that
\[\vert v\vert+\vert v^\prime\vert\leq C,\]
for some $C>0$.
Then
\begin{equation}\label{L2mpliesLinfty1}
\|v\|_{L^\infty(\R;\R^m)}\leq(\frac{3}{2}C)^\frac{1}{3}\|v\|^\frac{2}{3}.
\end{equation}
\end{lemma}
\begin{proof}
The assumptions on $v$ imply the existence of $s_m\in\R$ such that
\[\vert v(s_m)\vert=m:=\|v\|_{L^\infty(\R;\R^m)}.\]
This and $\vert v^\prime\vert\leq C$ imply
\[\vert v(s)\vert\geq m-C\vert s-s_m\vert\;\;s\in(s_m-\frac{m}{C},s_m+\frac{m}{C}).\]
This and a routine computation complete the proof.
\end{proof}
For later reference we note that if $v\in L^2(\R;\R^m)$ satisfies the assumptions in Lemma \ref{L2mpliesLinfty} we have, for $\|v\|\leq\frac{2^\frac{1}{2}}{3^\frac{1}{2}}C$,
\begin{equation}\label{NonlTermEst}
\begin{split}
&\vert\int_\R\Big(W(\bar{u}+v)-W(\bar{u})-W_u(\bar{u})\cdot v-\frac{1}{2}W_{uu}(\bar{u})v\cdot v\Big){d} s\vert
\leq C_w\|v\|^\frac{8}{3},\\
&\vert\int_\R\Big(W_u(\bar{u}+v)\cdot v-W_u(\bar{u})\cdot v-W_{uu}(\bar{u})v\cdot v\Big){d} s\vert
\leq C_w\|v\|^\frac{8}{3},\\
&\vert\int_\R\Big(W_{uu}(\bar{u}+v)-W_{uu}(\bar{u})\Big)v\cdot v{d} s\vert
\leq C_w\|v\|^\frac{8}{3}.
\end{split}
\end{equation}
for some constant $C_w>0$.
Since $W$ is a $C^3$ function and $\vert v\vert\leq C$
\[\max_{\vert z\vert\leq C}\vert W_{uuu}(\bar{u}+z)\vert\leq\text{Const}.\]
This and Taylor's formula imply that, for each $s\in\R$, it results
\[\vert W(\bar{u}+v)-W(\bar{u})-W_u(\bar{u})\cdot v-\frac{1}{2}W_{uu}(\bar{u})v\cdot v\vert\leq C\vert v\vert^3
\leq C\|v\|^\frac{2}{3}\vert v\vert^2,\]
where, for the last inequality, we have used (\ref{L2mpliesLinfty1}).
Then (\ref{NonlTermEst})$_1$ follows by integrating on $\R$. The other inequalities are proved in the same way.

 If $v\in W^{1,2}(\R;\R^m)$, $v\neq 0$ we sometime write $v$ in the form
\[v(s)=q\nu(s),\;\;s\in\R,\]
where $q=\|v\|$ is the $L^2$ norm of $v$, and $\nu\in\{w\in W^{1,2}(\R,\R^m):\|w\|=1\}$.

\noindent
For  $\bar{u}\in\{\bar{u}_1,\ldots,\bar{u}_N\}$ fixed, the effective potential defined in (\ref{eff-def}) can be considered a function of $q\in\R$ and $\nu\in W^{1,2}(\R,\R^m):\|w\|=1\}$. To emphasize this point of view, sometime, we write $\mathcal{W}(q,\nu)$ instead of $\mathcal{W}(q\nu)$.
Recall that $\|v\|_1$ denotes the $W^{1,2}(\R,\R^m)$ norm of $v$.
\begin{lemma}\label{lemmaw-true}
Let $v\in W^{1,2}(\R;\R^m)$ be as in Lemma \ref{L2mpliesLinfty} and assume that
\[\langle v,\bar{u}^\prime\rangle=0.\]
 Then the constant $q^0$ in Lemma \ref{lemmaw} can be chosen so that the effective potential
  $\mathcal{W}(q,\nu)$ is increasing in $q$ for $q\in[0,q^0]$ and there is $\mu>0$ such that
\begin{equation}\label{W2geqq2}
\frac{\partial^2}{\partial q^2}\mathcal{W}(q,\nu)\geq\mu(1+\|\nu^\prime\|^2),\quad q\in(0,q^0],
\end{equation}
and
\begin{equation}\label{W2geqq2-i}
\begin{split}
&\mathcal{W}(q,\nu)\geq\frac{1}{2}\mu q^2(1+\|\nu^\prime\|^2),\quad q\in(0,q^0]\\
&\Leftrightarrow\\
&\mathcal{W}(v)\geq\frac{1}{2}\mu\|v\|_1^2,\quad\|v\|\in(0,q^0].
\end{split}
\end{equation}
Moreover it results
\begin{equation}\label{wef-tvv}
\vert\mathcal{W}(v)-\frac{1}{2}\langle Tv,v\rangle\vert\leq C\|v\|^\frac{8}{3},\;\;\|v\|\in(0,q^0],
\end{equation}
where $T$ is defined in (\ref{operator}).
\end{lemma}
\begin{proof}
We have
\begin{equation}\label{j-jbar}
\begin{split}
&\mathcal{W}(q,\nu)=J_\R(\bar{u}+q\nu)-J_\R(\bar{u})\\
&=\langle\bar{u}^\prime,q\nu^\prime\rangle+\frac{1}{2}q^2\|\nu^\prime\|^2+\int_\R\Big(W(\bar{u}+q\nu)-W(\bar{u})\Big){d} s.
\end{split}
\end{equation}

\noindent
By differentiating twice $\mathcal{W}(q,\nu)$ with respect to $q$ yields
\begin{equation}\label{44}
\begin{split}
&\frac{\partial^2}{\partial q^2}\mathcal{W}(q,\nu)=\|\nu^\prime\|^2+\int_\R W_{uu}(\bar{u}+q\nu)\nu\cdot\nu{d} s\\
&=\frac{\partial^2}{\partial q^2}\mathcal{W}(0,\nu)+\int_\R\Big(W_{uu}(\bar{u}+q\nu)-W_{uu}(\bar{u})\Big)\nu\cdot\nu{d} s.
\end{split}
\end{equation}
From $q\nu=v$, and (\ref{NonlTermEst}) we have
\begin{equation}\label{wu-wu}
\vert\int_\R\Big(W_{uu}(\bar{u}+q\nu)-W_{uu}(\bar{u})\Big)\nu\cdot\nu{d} s\vert\leq C q^\frac{2}{3}.
\end{equation}

\noindent
We now observe that
\begin{eqnarray}\label{d2-tnu}
\frac{\partial^2}{\partial q^2}\mathcal{W}(0,\nu)=\langle T\nu,\nu\rangle.
\end{eqnarray}
$T$ is a self-adjoint operator which is positive by the minimality of $\bar{u}$. From the assumption that $a_\pm$ are non degenerate  the matrix $W_{uu}(a_\pm)$ is positive definite and Theorem A.2 in \cite{he} implies that the essential spectrum of $T$ is bounded below by a positive constant $\mu_e>0$. Since  $\bar{u}^\prime$ is an eigenvector of $T$, the assumption that $0$ is a simple eigenvalue of $T$ implies that if $\mu_1<\mu_e$ is an eigenvalue of $T$ then
 \[\mu_1=\inf_{\langle\nu,\bar{u}^\prime\rangle=0}\langle T\nu,\nu\rangle>0.\] From this, (\ref{d2-tnu}) and Theorem 13.31 in \cite{r} it follows that there is $\mu_2>0$ such that
\begin{eqnarray}
\frac{\partial^2}{\partial q^2}\mathcal{W}(0,\nu)\geq\mu_2>0,
\end{eqnarray}

\noindent
which together with (\ref{wu-wu}) implies
\begin{equation}\label{d2-geq-mu}
\frac{\partial^2}{\partial q^2}\mathcal{W}(q,\nu)\geq\frac{\mu_2}{2},\quad\text{ for }\,q\in[0,q^0]
\end{equation}
provided we assume $q^0\leq(\frac{\mu_2}{2C})^\frac{3}{2}$.

\noindent
To upgrade this estimate to (\ref{W2geqq2})  we use a trick from \cite{cp}. Recalling (\ref{44}) we have
\[\begin{split}
&\frac{\partial^2}{\partial q^2}\mathcal{W}(q,\nu)-\mu(1+\|\nu^\prime\|^2)\\
&=(1-\mu)\Big(\|\nu^\prime\|^2+\int_\R W_{uu}(\bar{u}+q\nu)\nu\cdot\nu{d} s\Big)\\
&+\mu\int_\R\Big(W_{uu}(\bar{u}+q\nu)-I\Big)\nu\cdot\nu{d} s\\
&\geq(1-\mu)\frac{\mu_2}{2}-\mu\Big(\vert\int_\R\Big(W_{uu}(\bar{u}+q\nu)-W_{uu}(\bar{u})\Big)\nu\cdot\nu{d} s\vert\\
&+\vert\int_\R\Big(W_{uu}(\bar{u})-I\Big)\nu\cdot\nu{d} s\vert\Big)\\
&\geq(1-\mu)\frac{\mu_2}{2}-\mu(C q^\frac{2}{3}+C^\prime)\geq(1-\mu)\frac{\mu_2}{2}-2\mu C^\prime=0,\\
&\quad\text{ for }\;\;\mu=\frac{\mu_2}{\mu_2+4C^\prime},\;q\leq q^0\leq(\frac{C^\prime}{C})^\frac{3}{2}.
\end{split}
\]
where we have used (\ref{wu-wu}) and $\vert W_{uu}(\bar{u})-I\vert\leq C^\prime$. This concludes the proof of (\ref{W2geqq2}). The inequality in (\ref{W2geqq2-i}) follows from
 (\ref{W2geqq2}) and
\[\mathcal{W}(0,\nu)=\frac{\partial}{\partial q}\mathcal{W}(0,\nu)=0,\]
which is a consequence of the definition of $\mathcal{W}(q,\nu)$ and of the minimality of $\bar{u}$. To complete the proof we note that (\ref{wef-tvv}) follows from (\ref{NonlTermEst}) and (\ref{j-jbar}) that, after observing  $\langle\bar{u}^\prime,v^\prime\rangle=-\langle\bar{u}^{\prime\prime},v\rangle=-\langle W_u(\bar{u}),v\rangle$, can be rewritten as
\begin{equation}\label{j-jbar-v}
\begin{split}
&\mathcal{W}(v)=\langle\bar{u}^\prime,v^\prime\rangle+\frac{1}{2}\|v^\prime\|^2+\int_\R\Big(W(\bar{u}+q\nu)-W(\bar{u})\Big){d} s\\
&=\frac{1}{2}\langle Tv,v\rangle+\int_\R\Big(W(\bar{u}+q\nu)-W(\bar{u})-W_u(\bar{u})v-\frac{1}{2}W_{uu}(\bar{u})v\cdot v\Big){d} s
\end{split}
\end{equation}
 The proof is complete.
\end{proof}
Lemma \ref{lemmaw-true} describes the properties of the effective potential $\mathcal{W}$ in a neighborhood of one of the connections represented by $\bar{u}_1,\ldots,\bar{u}_N$. We also need a lower bound for the effective potential away from a neighborhood of the $N$ connections.

\begin{lemma}\label{away}
Let $q^0$ the constant in Lemma \ref{lemmaw}. For ${p}\in(0,q^0]$ let $\mathcal{U}_{p}$ be the set of $C^{2,\alpha}(\R;\R^m)$ maps that satisfy
\begin{enumerate}
\item\begin{equation}\label{smoothness1}
\|\mathrm{u}\|_{C^{2,\alpha}(\R;\R^m)}\leq C,
\end{equation}
for some constants $C$.
\item\begin{equation}\label{near-a2}
\vert\mathrm{u}(s)-a_{\sigma(s)}\vert\leq g(\vert s\vert),\quad\text{ for }\;\;s\in\R,
\end{equation}
where $g:[0,\infty)\rightarrow(0,\infty)$ is a positive decreasing function that converges to $0$ at infinity.
\item\begin{equation}\label{w12-distance}
\|\mathrm{u}-\bar{u}_i(\cdot-r)\|_1\geq {p},\;\;r\in\R,\;i=1,\ldots,N.
\end{equation}
\end{enumerate}
Then
there exists $e_{p}>0$ independent of $g$ such that
\[J_\R(\mathrm{u})-c_0\geq e_{p},\quad\text{ for }\;\;\mathrm{u}\in\mathcal{U}_{p}.\]
\end{lemma}
\begin{proof}
Assume instead that
\[\lim_{j\rightarrow+\infty}J_\R(\mathrm{u}_j)= c_0\]
along a sequence $\mathrm{u}_j\in\mathcal{U}_{p}$, $j=1,\ldots$.

Then we can also assume $J_\R(\mathrm{u}_j)\leq 2c_0$. From  this and the fact that (\ref{near-a2}) is a convex condition it follows that, possibly by passing to a subsequence,
\[\mathrm{u}_j\rightharpoonup {\rm{u}}, \text{ in } W^{1,2}_{\mathrm{loc}}(\R;\R^m),\]
for some $\mathrm{u}\in W^{1,2}_{\mathrm{loc}}(\R;\R^m)$ that satisfies (\ref{near-a2})
and
\[\lim_{j\rightarrow+\infty}J(\mathrm{u}_j)=J(\mathrm{u})= c_0.\]
This and the fact that $\bar{u}_1,\ldots,\bar{u}_N$ are, by assumption, (modulo translation) the only minimizers of $J_\R$, imply
\[\mathrm{u}=\bar{u}=\bar{u}_i(\cdot-s_0),\]
for some $i\in\{1,\ldots,N\}$ and $s_0\in\R$.
Since $a_\pm$ are non degenerate zeros of $W\geq 0$ and $W$ is a $C^3$ function, there exist positive constants $\gamma, \Gamma$ and $r_0>0$ such that
\begin{equation}\label{second-derW}
\begin{split}
&W_{uu}(a_\pm+z)\zeta\cdot\zeta\geq\gamma^2\vert\zeta\vert^2,\;\;\zeta\in\R^m,\;\vert z\vert\leq r_0,\\
&\frac{1}{2}\gamma^2\vert z\vert^2\leq W(a_\pm+z)\leq\frac{1}{2}\Gamma^2\vert z\vert^2,\;\;\vert z\vert\leq r_0.
\end{split}
\end{equation}

If $I\subset\R$ is an interval we let $\|v\|_{1,I}$ the $W^{1,2}(I;\R^m)$ norm of $v$. Given $\tau>0$ we let $I_\tau=(-\tau,\tau)$, ${I}_\tau^+=[\tau,+\infty)$ and  ${I}_\tau^-=(-\infty,-\tau]$. From (\ref{n-bound}) we can fix $\tau$ so that

\begin{equation}\label{energy-est}
J_{I_\tau^-}(\bar{u})+J_{I_\tau^+}(\bar{u})\leq\epsilon,
\end{equation}
and
\begin{equation}\label{W12-est}
\|\bar{u}-a_-\|_{1,I_\tau^-}+\|\bar{u}-a_+\|_{1,I_\tau^+}\leq\epsilon
\end{equation}
where $\epsilon>0$ is a small number to be chosen later.
Moreover, since both $\bar{u}$ and $\mathrm{u}_j$ satisfy  (\ref{near-a2}) we can take $\tau>0$ so large that
\begin{equation}\label{into-erre0}
\begin{split}
&\vert\mathrm{u}_j-a_-\vert\leq r_0,\quad\text{ in }\;I_\tau^-,\\
&\vert\mathrm{u}_j-a_+\vert\leq r_0,\quad\text{ in }\;I_\tau^+
\end{split}
\end{equation}
where $r_0$ is the constant in (\ref{second-derW}).
From (\ref{smoothness1}) we can assume that $\mathrm{u}_j$ converges in compact intervals in the $C^1$ sense to the limit map $\bar{u}$  and therefore that, for large $j$
\begin{equation}\label{for-largej}
\begin{split}
&\vert J_{I_\tau}(\mathrm{u}_j)-J_{I_\tau}(\bar{u})\vert\leq\epsilon,\\
&\|\mathrm{u}_j-\bar{u}\|_{1,I_\tau}\leq\epsilon.
\end{split}
\end{equation}

From (\ref{W12-est}), (\ref{for-largej})$_2$ and (\ref{w12-distance}) it follows, for $\epsilon>0$ small
\begin{equation}\label{for-largej1}
\begin{split}
&\|\mathrm{u}_j-a_-\|_{1,{I}_\tau^-}+\|\mathrm{u}_j-a_+\|_{1,{I}_\tau^+}\\
&\geq
\|\mathrm{u}_j-\bar{u}\|_{1,I_\tau^-}+\|\mathrm{u}_j-\bar{u}\|_{1,I_\tau^+}
-\|\bar{u}-a_-\|_{1,I_\tau^-}-\|\bar{u}-a_+\|_{1,I_\tau^+}\\
&\geq\|\mathrm{u}_j-\bar{u}\|_1-\|\mathrm{u}_j-\bar{u}\|_{1,I_\tau}
-\|\bar{u}-a_-\|_{1,I_\tau^-}-\|\bar{u}-a_+\|_{1,I_\tau^+}\\
&\geq p-3\epsilon\geq \frac{p}{2}.
\end{split}
\end{equation}
 From (\ref{for-largej1}), using also (\ref{second-derW}) and (\ref{into-erre0}), we obtain
\[
\begin{split}
&\frac{p^2}{8}\leq\|\mathrm{u}_j-a_-\|_{1,{I}_\tau^-}^2+\|\mathrm{u}_j-a_+\|_{1,{I}_\tau^+}^2\\
&\leq\int_{{I}_\tau^-}(\vert \mathrm{u}_j-a_-\vert^2+\vert \mathrm{u}_j^\prime\vert^2){d} s
+\int_{{I}_\tau^+}(\vert \mathrm{u}_j-a_+\vert^2+\vert \mathrm{u}_j^\prime\vert^2){d} s\\
&\leq C_\gamma\int_{{I}_\tau^-}\Big(\gamma^2\vert \mathrm{u}_j-a_-\vert^2+\frac{1}{2}\vert \mathrm{u}_j^\prime\vert^2\Big){d}\tau+C_\gamma\int_{{I}_\tau^+}\Big(\gamma^2\vert\mathrm{u}_j-a_+\vert^2+\frac{1}{2}\vert \mathrm{u}_j^\prime\vert^2{d}\tau\Big){d}\tau\\
&\leq  C_\gamma(J_{{I}_s^-}(\mathrm{u}_j)+J_{{I}_s^+}(\mathrm{u}_j))
\end{split}
\]
where $C_\gamma=\max\{\frac{1}{\gamma^2},2\}$.
From this (\ref{energy-est}) and (\ref{for-largej}) it follows, for $j$ large and $\epsilon$ small
\begin{equation}\label{contra-lim}
\begin{split}
&J(\mathrm{u}_j)=J_{I_\tau}(\mathrm{u}_j)+J_{{I}_\tau^-}(\mathrm{u}_j)+J_{{I}_\tau^+}(\mathrm{u}_j)\\
&\geq J_{I_\tau}(\bar{u})-\epsilon+\frac{p^2}{8C_\gamma}\\
&\geq c_0-2\epsilon+\frac{p^2}{8C_\gamma}\\
&\geq c_0+\frac{p^2}{16C_\gamma}.
\end{split}
\end{equation}
 The estimate (\ref{contra-lim}) contradicts the minimizing character of the sequence $\mathrm{u}_j$ and concludes the proof.
\end{proof}

\subsection{Hamiltonian identities and a representation formula for the energy}\label{section23}

In this section we derive two identities that are consequences of the minimality of $u$ and are basic for our  proof of Theorem \ref{main} and Theorem \ref{scha-th}. These identities were noted (see Lemma 8.2 in \cite{scha}) but not exploited in \cite{scha}. The first identity, already considered in \cite{gui} and \cite{a} generalizes to the present PDE setting the classical theorem of conservation of mechanical energy. The other identity expresses an approximate orthogonality condition which does not have a finite dimensional counterpart. Following \cite{gui} we refer to these identities as Hamiltonian identities.
\begin{lemma}\label{properties}
Let $I\subset\R$ an interval and assume $u:I\times\R\rightarrow\R^m$ is a minimizer. Then
there exist constants $\omega$ and $\tilde{\omega}$ such that, for $x\in I$, it results
\begin{equation}\label{hamilton}
\begin{split}
&\int_\R\frac{1}{2}\vert u_x(x,y)\vert^2{d} y=\int_\R\Big(W(u(x,y))+\frac{1}{2}\vert u_y(x,y)\vert^2)\Big){d} y-c_0+\omega,\\
&\Leftrightarrow\\
&\frac{1}{2}\|u_x(x,\cdot)\|^2=J_\R(u(x,\cdot))-c_0+\omega.
\end{split}
\end{equation}
\begin{equation}\label{uxuyconst}
\int_\R u_x(x,y)\cdot u_y(x,y){d} y=\tilde{\omega},\quad\text{ for }\;x\in I.
\end{equation}
\end{lemma}
\begin{proof}
Given $[x_0,x_0+L]\subset I$ let $g:[x_0,x_0+L]\rightarrow[x_0,x_0+L]$ a continuous increasing surjection with inverse $\gamma:[x_0,x_0+L]\rightarrow[x_0,x_0+L]$. Define
\[v(s,y)=u(g(s),y),\quad\text{ for }\;s\in[x_0,x_0+L],\;y\in\R.\]
Then the energy $\mathcal{J}_{\mathcal{R}_L(x_0)}(v)$ of $v$ in the strip $\mathcal{R}_L(x_0):=(x_0,x_0+L)\times\R$ is given by
\begin{equation}\label{jv}
\begin{split}
&\mathcal{J}_{\mathcal{R}_L(x_0)}(v)=\int_{x_0}^{x_0+L}\int_\R\Big(W(v(s,y)+\frac{1}{2}(\vert v_s(s,y)\vert^2+\vert v_y(s,y)\vert^2\Big){d} y{d} s\\
&=\int_{x_0}^{x_0+L}\gamma^\prime(x)\int_\R(W(u(x,y)+\frac{1}{2}\vert u_y(x,y)\vert^2){d} y{d} x\\
&+\int_{x_0}^{x_0+L}\frac{1}{\gamma^\prime(x)}\int_\R\frac{1}{2}\vert u_x(x,y)\vert^2{d} y{d} x.
\end{split}
\end{equation}
The minimality of $u$ implies
\begin{equation}\label{min-R}
\mathcal{J}_{\mathcal{R}_L(x_0)}(v)\geq\mathcal{J}_{\mathcal{R}_L(x_0)}(u),
\end{equation}
for all $\gamma$. Note that, since $\mathcal{R}_L(x_0)$ is unbounded, to state (\ref{min-R}) we need to invoke Lemma \ref{point-exp} and Lemma \ref{extend}. From (\ref{min-R}), if we set $\gamma(x)=x+\lambda\sigma(x)$ with $\sigma$ an arbitrary $C^1$ function that satisfies $\sigma(x_0)=\sigma(x_0+L)=0$, we obtain
\begin{equation}\label{arbitrary}
\begin{split}
&0=\frac{d}{d\lambda}\mathcal{J}_{\mathcal{R}_L(x_0)}(v)\vert_{\lambda=0}\\
&=\int_{x_0}^{x_0+L}\int_\R\Big(W(u(x,y)+\frac{1}{2}\vert u_y(x,y)\vert^2-\frac{1}{2}\vert u_x(x,y)\vert^2)\Big)\sigma^\prime(x){d} y{d} x.
\end{split}
\end{equation}
\noindent
 Since $\sigma^\prime$ is an arbitrary function with zero average and (\ref{arbitrary}) holds for every $x_0$ and every $L>0$ we obtain (\ref{hamilton}).

Let
 $g:[x_0,x_0+L]\rightarrow\R$ a $C^1$ function that satisfies
 \begin{equation}\label{0boundary}
 g(x_0)=g(x_0+L)=0.
 \end{equation}
 \noindent
Define
\[v^\lambda(x,y)=u(x,y-\lambda g(x)),\quad\text{ for }\;x\in[x_0,x_0+L],\;y\in\R,\;\lambda\in(-1,1).\]
Then
\begin{equation}\label{jv-1}
\begin{split}
&\mathcal{J}_{\mathcal{R}_L(x_0)}(v^\lambda)=\int_{x_0}^{x_0+L}\int_\R\Big(W(v^\lambda)+\frac{1}{2}(\vert v_x^\lambda\vert^2+\vert v_y^\lambda\vert^2\Big){d} y{d} x\\
&=\int_{x_0}^{x_0+L}\int_\R\Big(W(u)+\frac{1}{2}\vert u_x\vert^2+(1+\lambda^2\vert g^\prime\vert^2)\frac{1}{2}\vert u_y\vert^2
-\lambda g^\prime u_x\cdot u_y\Big){d} y{d} x.
\end{split}
\end{equation}
From (\ref{0boundary}) we have $v(x_0,\cdot)=u(x_0,\cdot)$ and $v(x_0+L,\cdot)=u(x_0+L,\cdot)$. This and the minimality of $u$ imply
\[0=\frac{d}{d\lambda}J_{\mathcal{R}_L(x_0)}(v^\lambda)\vert_{\lambda=0}=\int_{x_0}^{x_0+L} g^\prime\int_\R u_x\cdot u_y{d} y{d} x\]
for all $g^\prime$ such that $\int_{x_0}^{x_0+L}g^\prime{d} x=0$.
This proves (\ref{uxuyconst}) for some constant $\tilde{\omega}$.
The proof is complete.
\end{proof}
If, as suggested in the Introduction, we regard the minimizer $u:\R^2\rightarrow\R^m$ as a map from $\R$ to $\bar{u}+W^{1,2}(\R;\R^m)$ and interpret $x$ as time, we see that, as in classical mechanics, the identity (\ref{hamilton}) says exactly that the sum of kinetic energy $\frac{1}{2}\int_R\vert u_x\vert^2{d} y$ and potential energy, the negative of the effective potential, is a constant of the motion.

If $I=\R$ we have $\omega=\tilde{\omega}=0$.
\begin{lemma}\label{lemma=0}
Assume $u:\R^2\rightarrow\R^m$ is the minimizer in Theorem \ref{main}. Then
\begin{equation}\label{Hamilton=0}
\begin{split}
&\int_\R\frac{1}{2}\vert u_x\vert^2{d} y=\int_\R\Big(W(u)+\frac{1}{2}\vert u_y\vert^2\Big){d} y-c_0,\;\;x\in\R,\\
&\Leftrightarrow\\
&\frac{1}{2}\|u_x(x,\cdot)\|^2=J_\R(u(x,\cdot))-c_0,\;\;x\in\R.
\end{split}
\end{equation}
\begin{equation}\label{InnerProduct=0}
\int_\R u_x\cdot u_y{d} y=0,\;\;x\in\R,
\end{equation}
\end{lemma}
\begin{proof}
From (\ref{hamilton}) and (\ref{KineticLim}) in Lemma \ref{kineticUpperBound} it follows
\[\lim_{x\rightarrow+\infty}\int_\R\Big(W(u)+\frac{1}{2}\vert u_y\vert^2\Big){d} y-c_0=-\omega\geq 0.\]
If $-\omega=\vert\omega\vert>0$ there exists $x_\omega$ such that
\[\int_\R\Big(W(u)+\frac{1}{2}\vert u_y\vert^2\Big){d} y\geq c_0+\frac{\vert\omega\vert}{2},\;\;x\geq x_\omega.\]
and therefore (\ref{upper-bound-1}) in Lemma \ref{upper-bound} yields
\[\begin{split}
&(c_0+\frac{\vert\omega\vert}{2})(x-x_\omega)\leq\int_{x_\omega}^x\int_\R\Big(W(u)+\frac{1}{2}\vert u_y\vert^2\Big){d} y{d} x\\
&\leq\mathcal{J}_{(x-x_\omega)\times\R}(u)\leq c_0(x-x_\omega)+C_0,\;\;x\geq x_\omega.
\end{split}\]
which is incompatible with the assumption $-\omega>0$. This establishes (\ref{Hamilton=0}). To prove (\ref{InnerProduct=0}) note that from (\ref{uxuyconst}) and (\ref{KineticLim}) it follows
\[\vert\tilde{\omega}\vert=\lim_{x\rightarrow+\infty}\vert\int_\R u_x\cdot u_y{d} y\vert
\leq C\lim_{x\rightarrow+\infty}\|u_x(x,\cdot)\|=0,\]
where we have also used that (\ref{l2-der-norm}) and (\ref{n-bound}) imply $\|u_y\|\leq C$. The proof is complete.
\end{proof}

\vskip.2cm

We can now derive a special representation formula for the \emph{kinetic energy} $\frac{1}{2}\int_\R\vert u_x\vert^2{d} y$ of $u$.
Note that by differentiating the identities
\begin{equation}\label{=1}
\begin{split}
&0=\langle v(x,\cdot),\bar{u}^\prime\rangle=\langle v(x,\cdot-\tau),\bar{u}^\prime(\cdot-\tau)\rangle,\\
&\|v(x,\cdot)\|^2=\|v(x,\cdot-\tau)\|^2,\\
&1=\|\nu(x,\cdot)\|^2=\|\nu(x,\cdot-\tau)\|^2,
\end{split}
\end{equation}
and recalling the definition $v=q\nu$ of $\nu$ for $q>0$ we obtain
\begin{equation}\label{identity}
\begin{split}
&\langle v_x(x,\cdot),\bar{u}^\prime\rangle=0,\\
&\|v_y(x,\cdot)\|^2+\langle v(x,\cdot),v_{yy}(x,\cdot)\rangle=0,\\
&\langle\nu_x(x,\cdot),\nu(x,\cdot)\rangle=\langle\nu_y(x,\cdot),\nu(x,\cdot)\rangle
=\langle\nu_x(x,\cdot),\bar{u}^\prime\rangle=0,\\
\end{split}
\end{equation}
\begin{lemma}\label{represent}
Assume $u:I\times\R\rightarrow\R^m$, $q^0$ and $q(x)\leq q^0$ as in Lemma \ref{hprime-hpprime?}. Then, if $q^0>0$ is sufficiently small, it results
\begin{equation}\label{hprime}
h^\prime(x)=\frac{\langle v_x(x,\cdot),v_y\rangle(x,\cdot)}{\|\bar{u}^\prime+v_y(x,\cdot)\|^2}
=\frac{q^2(x)\langle\nu_x(x,\cdot),\nu_y(x,\cdot)\rangle}{\|\bar{u}^\prime+q(x)\nu_y(x,\cdot)\|^2},
\end{equation}
and
\begin{equation}\label{xenergy}
\begin{split}
&\|u_x(x,\cdot)\|^2=\|v_x(x,\cdot)\|^2-\frac{\langle v_x(x,\cdot),v_y(x,\cdot)\rangle^2}{\|\bar{u}^\prime+v_y(x,\cdot)\|^2}\\
&=q_x^2(x)+q^2(x)\|\nu_x(x,\cdot)\|^2
-q^4(x)\frac{\langle\nu_x(x,\cdot),\nu_y(x,\cdot)\rangle^2}{\|\bar{u}^\prime+q(x)\nu_y(x,\cdot)\|^2}.
\end{split}
\end{equation}
Moreover the map
\[(0,q(x)]\ni p\rightarrow f(p,x)\|\nu_x(x,\cdot)\|^2:= p^2\|\nu_x(x,\cdot)\|^2-p^4\frac{\langle\nu_x(x,\cdot),\nu_y(x,\cdot)\rangle^2}{\|\bar{u}^\prime+p\nu_y(x,\cdot)\|^2}\]
is nonnegative and nondecreasing for each fixed $x\in I$.
\end{lemma}
\begin{proof}
From (\ref{decomp-i1}) we obtain
\[\begin{split}
& u_x(x,\cdot)=-h^\prime(x)\Big(\bar{u}^\prime(\cdot-h(x))+v_y(x,\cdot-h(x))\Big)+v_x(x,\cdot-h(x)),\\
& u_y(x,\cdot)=\bar{u}^\prime(\cdot-h(x))+v_y(\cdot-h(x)).
\end{split}\]
and therefore (\ref{InnerProduct=0}) in Lemma \ref{lemma=0} and (\ref{identity}) imply
\begin{equation}\label{hprime1}
\begin{split}
&0=\langle u_x(x,\cdot),u_y(x,\cdot)\rangle=
-h^\prime(x)(\|\bar{u}^\prime+v_y(x,\cdot)\|^2+\langle v_x(x,\cdot),v_y(x,\cdot)\rangle.
\end{split}
\end{equation}
From (\ref{l2-der-norm1?}) and (\ref{identity}) it follows
\begin{equation}\label{q-nuy}
\|v_y(x,\dot)\|^2\leq\|v(x,\cdot)\|\|v_{yy}(x,\dot)\|\leq\frac{K_1}{\sqrt{k_1}}\|v(x,\cdot)\|\leq\frac{K_1}{\sqrt{k_1}}q^0,
\end{equation}
and in turn, provided $q^0$ is sufficiently small,

\begin{equation}\label{den-bounds}
\frac{1}{2}\|\bar{u}^\prime\|\leq\|\bar{u}^\prime+v_y(x,\cdot)\|\leq\frac{3}{2}\|\bar{u}^\prime\|.
\end{equation}
Therefore (\ref{hprime1}) can be solved for $h^\prime$ and the first expression of $h^\prime$ in (\ref{hprime}) is established. The other expression follows by (\ref{identity}) which implies $\langle v_x,v_y\rangle
=\langle q_x\nu+q\nu_x,q\nu_y\rangle=q^2\langle\nu_x,\nu_y\rangle$.
A similar computation that also uses (\ref{hprime}) yields (\ref{xenergy}).

It remains to prove the monotonicity of $f(p,\cdot)\|\nu_x\|^2$. We can assume $\|\nu_x\|>0$ otherwise there is nothing to be proved. We have, using also (\ref{q-nuy}) and (\ref{den-bounds}),
\[\begin{split}
&D_pf(p,\cdot)=2p-4p^3\frac{\langle\frac{\nu_x}{\|\nu_x\|},\nu_y\rangle^2}{\|\bar{u}^\prime+p\nu_y\|^2}
+2p^4\frac{\langle\frac{\nu_x}{\|\nu_x\|},\nu_y\rangle^2\langle\bar{u}^\prime+p\nu_y,\nu_y\rangle}{\|\bar{u}^\prime+p\nu_y\|^4}\\
&\geq p\Big(2-p^2\|\nu_y\|^2\frac{16}{\|\bar{u}^\prime\|^2}-p^3\|\nu_y\|^3\frac{16}{\|\bar{u}^\prime\|^3}\Big)\\
&\geq p\Big(2-q^0C^2\frac{16}{\|\bar{u}^\prime\|^2}
-(q^0)^\frac{3}{2}C^3\frac{16}{\|\bar{u}^\prime\|^3}\Big).
\end{split}\]
This proves $D_pf(p,\cdot)>0$ for $q^0\leq\frac{\|\bar{u}^\prime\|^2}{32 C^2}$.
The proof is complete.
\end{proof}

\subsection{Completing the proof of Theorem \ref{main}}\label{section24}

From (\ref{Hamilton=0}) and (\ref{KineticLim}) it follows
the existence of $x_0\in\R$ such that
\begin{equation}\label{WeffBound}
J_\R(u(x,\cdot))-c_0<\frac{e_{q^0}}{2},\;\;x\geq x_0
\end{equation}
and Lemma \ref{away} and $\|\cdot\|\leq\|\cdot\|_1$  imply
\begin{equation}\label{inNeig}
\min_j\min_{r\in\R}\|u(x,\cdot)-\bar{u}_j(\cdot-r)\|<q^0,\;\;x\geq x_0.
\end{equation}
Since we assume $q^0<q^*$, with $q^*>0$ defined in (\ref{q*}) as the minimal distance between any two distinct $\bar{u}_i$ and $\bar{u}_j$ and their translates, (\ref{inNeig}) says that, for $x\geq x_0$, $u(x,\cdot)$ is trapped in a neighborhood of one of the connections and Lemma \ref{lemmaw} implies the existence of uniquely determined $\bar{u}_+\in\{\bar{u}_1,\ldots,\bar{u}_N\}$, independent of $x\geq x_0$, and of a function $h:[x_0,+\infty)\rightarrow\R$ such that
\begin{equation}\label{NewOrt}
\begin{split}
&\|v(x,\cdot)\|=\|u(x,\cdot)-\bar{u}_+(\cdot-h(x))\|=\min_j\min_r\|u(x,\cdot)-\bar{u}_j(\cdot-r)\|,\\
& \langle v(x,\cdot),\bar{u}_+^\prime\rangle=\langle u(x,\cdot)-\bar{u}_+(\cdot-h(x)),\bar{u}_+^\prime(\cdot-h(x))\rangle=0,
\end{split}
\end{equation}
where $v(x,y)=u(x,y+h(x))-\bar{u}_+(y)$.

Note that (\ref{inNeig}) implies $q(x)=\|v(x,\cdot)\|<q^0$ and therefore, for $x\geq x_0$, we have that $u(x,\cdot)$ remains in the \emph{convex region} of the effective potential where (\ref{W2geqq2}) holds. We can expect that this implies exponential decay of $u(x,\cdot)$ to a translate of $\bar{u}_+$. We have indeed
\begin{lemma}\label{to-baru+}
There exist $k,C>0$ and $x_+, \eta_+\in\R$ such that
\begin{equation}\label{q(x)-0}
q(x)=\|v(x,\cdot)\|\leq\sqrt{2}q^0e^{-k(x-x_+)},\;\;x\geq x_+.
\end{equation}
and
\begin{equation}\label{u(x)-baru}
\|u(x,\cdot)-\bar{u}_+(\cdot-\eta_+)\|\leq C(q^0)^\frac{1}{2}e^{-\frac{k}{2}(x-x_+)},\;\;x\geq x_+.
\end{equation}
Analogous statements apply to the interval $(-\infty,x_-]$ for some $x_-,\eta_-\in\R$ and some
$\bar{u}_-\in\{\bar{u}_1,\ldots,\bar{u}_N\}$.
\end{lemma}
Before giving the proof we remark on the different meaning of (\ref{q(x)-0}) and (\ref{u(x)-baru}). Equation
(\ref{q(x)-0}) says that, as $x\rightarrow+\infty$, $u(x,\cdot)$ converges exponentially to the manifold of the translates of $\bar{u}_+$ while (\ref{u(x)-baru}) implies convergence to a specific element of that manifold.
\begin{proof}

\noindent
1. There is $x_0\in\R$ such that
\begin{equation}\label{ElIn0}
\frac{d^2}{dx^2}\|v(x,\cdot)\|^2\geq\frac{1}{2}\mu\|v(x,\cdot)\|^2,\;\;x\geq x_0,
\end{equation}
where $\mu>0$ is the constant in (\ref{W2geqq2}).
To show this we begin by the elementary inequality
\begin{equation}\label{ElIn}
\begin{split}
&\frac{d^2}{dx^2}\|v(x,\cdot)\|^2=\frac{d^2}{dx^2}\|u(x,\cdot)-\bar{u}_+(\cdot-h(x))\|^2\\
&\geq2\Big\langle\frac{d^2}{dx^2}\Big(u(x,\cdot)-\bar{u}_+(\cdot-h(x))\Big),u(x,\cdot)-\bar{u}_+(\cdot-h(x))\Big\rangle.
\end{split}
\end{equation}
From
\[
\begin{split}
&\frac{d^2}{dx^2}\Big(u(x,\cdot)-\bar{u}_+(\cdot-h(x))\Big)\\
&=u_{xx}(x,\cdot)-\bar{u}_+^{\prime\prime}(\cdot-h(x))(h^\prime(x))^2+\bar{u}_+^\prime(\cdot-h(x))h^{\prime\prime}(x),
\end{split}
\]
and (\ref{ElIn}), using also (\ref{NewOrt}) and (\ref{hprime}), it follows
\begin{equation}\label{ElIn1}
\begin{split}
&\frac{d^2}{dx^2}\|v(x,\cdot)\|^2\geq2\langle u_{xx}(x,\cdot),u(x,\cdot)-\bar{u}_+(\cdot-h(x))\rangle\\
&-2\langle \bar{u}_+^{\prime\prime},v(x,\cdot)\rangle\frac{\langle v_x(x,\cdot),v_y(x,\cdot)\rangle^2}{\|\bar{u}_+^\prime+v_y(x,\cdot)\|^4}
=2I_1+2I_2,
\end{split}
\end{equation}
where we have also used $\langle f(\cdot-\tau),g(\cdot-\tau)\rangle=\langle f,g\rangle$.
Since both $u$ and $\bar{u}_+$ solve (\ref{system}) we have
\[
u_{xx}(x,\cdot)=W_u(u(x,\cdot))-W_u(\bar{u}_+(\cdot-h(x)))-\Big(u(x,\cdot)-\bar{u}_+(\cdot-h(x))\Big)_{yy}.
\]
Then, recalling the definition (\ref{operator}) of $T$ and that $v(x,\cdot)=u(x,\cdot+h(x))-\bar{u}_+$, after an integration by parts, we obtain
\[\begin{split}
&I_1=\Big\langle W_u(u(x,\cdot))-W_u(\bar{u}_+(\cdot-h(x)))-\Big(u(x,\cdot)-\bar{u}_+(\cdot-h(x))\Big)_{yy},\\
&\quad\quad u(x,\cdot)-\bar{u}_+(\cdot-h(x))\Big\rangle\\
&=\langle W_u(\bar{u}_++v(x,\cdot))-W_u(\bar{u}_+)-v_{yy}(x,\cdot),v(x,\cdot)\rangle\\
&=\langle W_u(\bar{u}_++v(x,\cdot))-W_u(\bar{u}_+),v(x,\cdot)\rangle+\|v_y(x,\cdot)\|^2\\
&=\langle W_u(\bar{u}_++v(x,\cdot))-W_u(\bar{u}_+)-W_{uu}(\bar{u}_+)v(x,\cdot),v(x,\cdot)\rangle
+\langle Tv(x,\cdot),v(x,\cdot)\rangle.
\end{split}\]
This, (\ref{NonlTermEst}) and (\ref{wef-tvv}) imply
\begin{equation}\label{I1}
I_1\geq 2\mathcal{W}(v(x,\cdot))-C\|v\|^\frac{8}{3},\;\;x\geq x_0.
\end{equation}
To estimate $I_2$ we note that
(\ref{q-nuy}) implies that, provided $q^0$ is sufficiently small, we can assume (\ref{den-bounds}) (with $\bar{u}=\bar{u}_+$) together with
\begin{equation}\label{vy-small}
\frac{\langle v_x(x,\cdot),v_y(x,\cdot)\rangle^2}{\|\bar{u}^\prime+v_y(x,\cdot)\|^2}
\leq\frac{1}{2}\|v_x(x,\cdot)\|^2,\,\;x\geq x_0.
\end{equation}
Then (\ref{xenergy}) and (\ref{Hamilton=0}) imply
\begin{equation}\label{UxxUpper}
\|v_x(x,\cdot)\|^2\leq 4\mathcal{W}(v(x,\cdot)),\;\;x\geq x_0,
\end{equation}
and we obtain
\[I_2\leq C\|v(x,\cdot)\|\mathcal{W}(v(x,\cdot)),\;\;x\geq x_0.\]
From this and (\ref{I1}), using also (\ref{W2geqq2-i}) and $\|v(x,\cdot)\|\leq q^0$, we conclude
\[
\begin{split}
&I_1+I_2\geq(2-C\|v(x,\cdot)\|)\mathcal{W}(v(x,\cdot))-C\|v(x,\cdot)\|^\frac{8}{3}\\
&\geq\frac{1}{4}\mu\|v(x,\cdot)\|^2,\;\;x\geq x_0,
\end{split}
\]
and (\ref{ElIn0}) follows from (\ref{ElIn1}).

\noindent
2. Since from (\ref{inNeig}) we have $\|v(x,\cdot)\|\leq q^0$ for $x\geq x_0$, from 1. and the maximum principle we get, for every $l>0$
\begin{equation}\label{MaxPri}
\|v(x,\cdot)\|^2\leq\varphi_l(x),\;\;x\in[x_0,x_0+2l],
\end{equation}
where
\[\varphi_l(x):= (q^0)^2\frac{\cosh{\sqrt{\frac{\mu}{2}}(l-(x-x_0))}}{\cosh{\sqrt{\frac{\mu}{2}}l}},\;\;x\in(x_0,x_0+2l),\]
is the solution of the problem
\[\left\{\begin{array}{l}
\varphi^{\prime\prime}=\frac{\mu}{2}\varphi,\;\;x\in(x_0,x_0+2l),\\\\
\varphi(x_0)=\varphi(x_0+2l)=(q^0)^2.
\end{array}\right.\]
The estimate (\ref{q(x)-0}), with $k=\frac{1}{2}\sqrt{\frac{\mu}{2}}$ and $x_+=x_0$, follows from (\ref{MaxPri}) and the inequality
\[\varphi_l(x)\leq 2(q^0)^2e^{-\sqrt{\frac{\mu}{2}}(x-x_0)},\;\;x\in[x_0,x_0+l],\]
which is valid for all $l>0$.

\noindent
3. It results
\begin{equation}\label{hprimeUpper}
\vert h^\prime(x)\vert\leq C\|v(x,\cdot)\|^\frac{1}{2},\;\;x\geq x_0.
\end{equation}
From (\ref{UxxUpper}) and $\mathcal{W}(v(x,\cdot))\leq e_{q^0}$ we have
\[\|v_x(x,\cdot)\|^2\leq 4e_{q^0},\;\;x\geq x_0.\]
Then 3. follows from (\ref{hprime}), (\ref{q-nuy}) and (\ref{den-bounds}).

\noindent
4. Point 3. and (\ref{q(x)-0}) imply
\[\vert h^\prime(x)\vert\leq C(q^0)^\frac{1}{2}e^{-\frac{k}{2}(x-x_+)},\;\;x\geq x_+.\]
Therefore there exists $\eta_+\in\R$ such that
\[\lim_{x\rightarrow+\infty}h(x)=\eta_+\] and the convergence is exponential
\begin{equation}\label{EtaHaExp}
\vert\eta_+-h(x)\vert\leq C(q^0)^\frac{1}{2}e^{-\frac{k}{2}(x-x_+)},\;\;x\geq x_+.
\end{equation}
Now observe that from
\[\begin{split}
&\vert\bar{u}_+(y-(h(x)-\eta_+))-\bar{u}_+(y)\vert\\
&\leq\int_0^1\vert\bar{u}^\prime_+(y-t(h(x)-\eta_+))\vert{d} t\vert h(x)-\eta_+\vert\\
&\leq\Big(\int_0^1\vert\bar{u}^\prime_+(y-t(h(x)-\eta_+))\vert^2{d} t\Big)^\frac{1}{2}\vert h(x)-\eta_+\vert,
\end{split}\]
it follows
\begin{equation}\label{EtaDiff}
\begin{split}
&\|\bar{u}_+(\cdot-(h(x)-\eta_+))-\bar{u}_+(y)\|^2\\
&\leq\int_\R\int_0^1\vert\bar{u}^\prime_+(y-t(h(x)-\eta_+))\vert^2{d} t{d} y\vert h(x)-\eta_+\vert^2\\
&=\int_0^1\int_\R\vert\bar{u}^\prime_+(y-t(h(x)-\eta_+))\vert^2{d} y{d} t\vert h(x)-\eta_+\vert^2\\
&=\|\bar{u}_+\|^2\vert h(x)-\eta_+\vert^2.
\end{split}
\end{equation}
We have
\[\begin{split}
&\|u(x,\cdot)-\bar{u}_+(\cdot-\eta_+)\|\\
&\leq\|u(x,\cdot)-\bar{u}_+(\cdot-h(x))\|+\|\bar{u}_+(\cdot-h(x))-\bar{u}_+(\cdot-\eta_+)\|\\
&=\|v(x,\cdot)\|+\|\bar{u}_+(\cdot-(h(x)-\eta_+))-\bar{u}_+\|.
\end{split}\]
This, (\ref{EtaDiff}), (\ref{EtaHaExp}) and (\ref{q(x)-0}) imply (\ref{u(x)-baru}) and conclude the proof.
\end{proof}
Continuing the proof of Theorem \ref{main} we note that we have already established (\ref{asymptotic-u})$_1$ in (\ref{FirstSta}) and that (\ref{asymptotic-u})$_2$ follows from (\ref{u(x)-baru}) and Lemma \ref{L2mpliesLinfty}.
It remain to discuss the case $\bar{u}_-=\bar{u}_+=\bar{u}$.

From the previous analysis we know that $u(x,\cdot)$ remains in a $q^0$ neighborhood of $\bar{u}(\cdot-\eta_-)$ in $(-\infty,x_-]$ and of $\bar{u}(\cdot-\eta_+)$ in $[x_+,+\infty)$. The problem is to analyse what happens in the interval $(x_-,x_+)$. We prove that for $u(x,\cdot)$ is more convenient to remain near the manifold of the translates of $\bar{u}$ also in $(x_-,x_+)$. Indeed we show that to travel away from this manifold and come back to it is more penalizing from the point of view of minimizing the energy.

In the following,
 for $x$ in certain subintervals of $(x_-,x_+)$, we use test functions of the form
\begin{equation}\label{test}
\hat{u}(x,y)=\bar{u}(y-\hat{h}(x))+\hat{q}(x)\nu(x,y-\hat{h}(x))
\end{equation}
for suitable choices of the functions $\hat{q}=\hat{q}(x)$ and $\hat{h}=\hat{h}(x)$. We always take $\hat{q}(x)\leq q(x)\leq q^0$. Note that in (\ref{test}) the direction vector $\nu$ is the one associated to $u$ in the decomposition (\ref{decomp-i1}).

\noindent
From (\ref{test}) it follows
\begin{equation}\label{hatux-square}
\int_\R\vert\hat{u}_x\vert^2=(\hat{h}^\prime)^2\|\bar{u}^\prime+\hat{q}\nu_y\|^2-2\hat{h}^\prime \hat{q}^2\langle\nu_x,\nu_y\rangle+\hat{q}_x^2+\hat{q}^2\|\nu_x\|^2.
\end{equation}
We choose the value of $\hat{h}$ that minimizes (\ref{hatux-square}) that is
\begin{equation}\label{hat-hprime}
\hat{h}^\prime=\hat{q}^2\frac{\langle\nu_x,\nu_y\rangle}{\|\bar{u}^\prime+\hat{q}\nu_y\|^2},
\end{equation}
then we get
\begin{equation}\label{hat-xenergy0}
\int_\R\vert\hat{u}_x\vert^2=\hat{q}_x^2+\hat{q}^2\|\nu_x\|^2-\hat{q}^4\frac{\langle\nu_x,\nu_y\rangle^2}{\|\bar{u}^\prime+\hat{q}\nu_y\|^2}.
\end{equation}
Therefore we obtain the expression of the energy density of the test map $\hat{u}$
\begin{equation}\label{hat-xenergy}
\begin{split}
&\int_\R\frac{1}{2}\vert\hat{u}_x\vert^2+\int_\R(W(\hat{u})+\frac{1}{2}\vert\hat{u}_y\vert^2)-c_0\\
&=\frac{1}{2}\Big(\hat{q}_x^2+\hat{q}^2\|\nu_x\|^2-\hat{q}^4\frac{\langle\nu_x,\nu_y\rangle^2}{\|\bar{u}^\prime+\hat{q}\nu_y\|^2}\Big)+\mathcal{W}(\hat{q},\nu).
\end{split}
\end{equation}
Note that, since we do not change the direction vector $\nu(x,\cdot)$, this  expression is completely determined once we fix the function $\hat{q}$.
\begin{lemma}\label{NoMax}
Let $I\subset\R$ be an interval and assume that the minimizer $u:\R^2\rightarrow\R^m$ satisfies
\[q(x)\leq q^0,\;\;x\in I.\]
Then the map $x\rightarrow q(x)$ can not have points of maximum in I meaning that there are no $x_1<x^*<x_2\in I$ such that
\[q(x_i)<q(x^*),\;\;i=1,2.\]
\end{lemma}
\begin{proof}
Assume instead that $x_1<x^*<x_2\in I$ with $q(x_i)<q(x^*),\;\;i=1,2$ exist. Since $q=q(x)$ is continuous we can assume $q(x^*)=\max_{x\in[x_1,x_2]}q(x)$ and, by restricting the interval $(x_1,x_2)$ if necessary, that
\[q_0=q(x_i)<q(x)\leq q(x^*),\;\;i=1,2,\;x\in(x_1,x_2)\]
for some $q_0\in(0,q^0)$ that satisfies the condition
\[q(x^*)\leq 2q_0.\]
We show that this is in contradiction with the minimality of $u$ by constructing a competing function $\tilde{u}$ defined as follows: in the interval $(-\infty,x_1)$ we take
\begin{equation}\label{compar1}
\tilde{u}(x,\cdot)=u(x,\cdot),\quad\text{ for }\;x\in(-\infty, x_1).
\end{equation}
In the interval $[x_1,x_2]$ we take:
\begin{equation}\label{compar2}
\tilde{u}(x,\cdot)=\hat{u}(x,\cdot),\quad\text{ with }\;\hat{q}(x)=2q_0-q(x),\quad\text{ for }\;x\in[x_1, x_2],
\end{equation}
where $\hat{u}$ is defined in (\ref{test}) with $\hat{q}=2q_0-q$ and $\hat{h}$ the solution of (\ref{hat-hprime}) with initial condition $\hat{h}(x_1)=h(x_1)$. With this definition $\tilde{u}$ is continuous at $x=x_1$. Indeed, since $\hat{q}(x_1)=2q_0-q(x_1)=q_0=q(x_1)$ we have
\[\begin{split}
&\hat{u}(x_1,y)=\bar{u}(y-\hat{h}(x_1))+\hat{q}(x_1)\nu(x,y-\hat{h}(x_1))\\
&=\bar{u}(y-h(x_1))+q(x_1)\nu(x,y-h(x_1))=u(x_1,y).
\end{split}\]
For $x=x_2$ we have instead
\[\begin{split}
&\hat{u}(x_2,y)=\bar{u}(y-\hat{h}(x_2))+\hat{q}(x_2)\nu(x,y-\hat{h}(x_2))\\
&=\bar{u}(y-\hat{h}(x_2))+q(x_2)\nu(x,y-\hat{h}(x_2))\\
&=\bar{u}(y-h(x_2)-(\hat{h}(x_2)-h(x_2)))+q(x_2)\nu(x,y-h(x_2)-(\hat{h}(x_2)-h(x_2)))\\
&=u(x_2,y-(\hat{h}(x_2)-h(x_2))).
\end{split}\]
That is, at $x=x_2$, the function $\hat{u}(x_2,\cdot)$ coincides with the translation $u(x,\cdot-(\hat{h}(x_2)-h(x_2)))$ of $u(x,\cdot)$
where
\[\hat{h}(x_2)-h(x_2)=\int_{x_1}^{x_2}(\hat{h}^\prime-h^\prime){d} x=\int_{x_1}^{x_2}
\Big(\frac{(2q_0-q)^2\langle\nu_x,\nu_y\rangle}
{\|\bar{u}^\prime+(2q_0-q)\nu_y\|^2}-\frac{q^2\langle\nu_x,\nu_y\rangle}{\|\bar{u}^\prime+q\nu_y\|^2}\Big){d} x.\]
To compensate for this translation it is natural to complete the definition of $\tilde{u}$ by setting
\begin{equation}\label{compar3}
\begin{split}
&\tilde{u}(x,\cdot)=u\Big(x,\cdot-(\hat{h}(x_2)-h(x_2))(1-\frac{x-x_2}{l})\Big),\;\;x\in(x_2,x_2+l],\\
&\tilde{u}(x,\cdot)=u(x,\cdot),\;\;x\in(x_2+l,+\infty)
\end{split}
\end{equation}
so that $\tilde{u}(x,\cdot)$ is continuous at $x_2+l$ and coincides with $u(x,\cdot-(\hat{h}(x_2)-h(x_2)))$ for $x=x_2$.
The idea here is that, for large $l>0$, the contribution of the interval $(x_2,x_2+l)$ to the difference of energy between $u$ and $\tilde{u}$ can be disregarded with respect to the contribution of the interval $(x_1,x_2)$. Proceeding as in the proof of Lemma \ref{properties} with $g:[x_2,x_2+l]\rightarrow\R$ the linear function
\[g(x)=(\hat{h}(x_2)-h(x_2))(1-\frac{x-x_2}{l}),\;\;x\in[x_2,x_2+l],\]
we arrive at equation (\ref{jv-1}) with $g^\prime=-\frac{\hat{h}(x_2)-h(x_2)}{l}$, $\lambda=1$ and without the last term which vanishes on the basis of Lemma \ref{lemma=0}
\begin{equation}\label{jv-11}
\begin{split}
&\mathcal{J}_{(x_2,x_2+l)\times\R}(\tilde{u})\\
&=\int_{x_2}^{x_2+l}\Big(\int_\R(W(u)+\frac{1}{2}\vert u_x\vert^2)dy+(1+\vert g^\prime\vert^2)\int_\R\frac{1}{2}\vert u_y\vert^2{d} y\Big){d} x\\
&=\mathcal{J}_{(x_2,x_2+l)\times\R}(u)+\int_{x_2}^{x_2+l}\vert g^\prime\vert^2\int_\R\frac{1}{2}\vert u_y\vert^2{d} y{d} x.
\end{split}
\end{equation}
It follows
\begin{equation}\label{jv-111}
\begin{split}
&\mathcal{J}_{(x_2,x_2+l)\times\R}(\tilde{u})
-\mathcal{J}_{(x_2,x_2+l)\times\R}(u)\\
&=\int_{x_2}^{x_2+l}\vert g^\prime\vert^2\int_\R\frac{1}{2}\vert u_y\vert^2{d} y{d} x\leq C\frac{(\hat{h}(x_2)-h(x_2))^2}{l}.
\end{split}
\end{equation}
where we have also used (\ref{l2-der-norm}) and (\ref{n-bound}).
From the definition (\ref{compar1}), (\ref{compar2}) and (\ref{compar3}) of $\tilde{u}$ we have that in $(x_1,x_2)$ it results $\hat{q}_x=(2q_0-q)_x=-q_x$ and therefore
\begin{equation}\label{hatq-x}
\hat{q}_x^2=q_x^2.
\end{equation}
Moreover from (\ref{q-nuy}) and Lemma \ref{represent} it follows
\[f(\hat{q}(x))\leq f(q(x)),\quad x\in(x_1,x_2).\]
From this (\ref{hatq-x}) and (\ref{W2geqq2}) in Lemma \ref{lemmaw-true} that implies the strict monotonicity of the map $q\rightarrow\mathcal{W}(q,\nu)$ we conclude
\[\mathcal{J}_{(x_1,x_2)\times\R}(\tilde{u})<\mathcal{J}_{(x_1,x_2)\times\R}(u).\]
This and (\ref{jv-111}), for $l>0$ sufficiently large, imply
\[\begin{split}
&\mathcal{J}_{(x_1,x_2+l)\times\R}(u)-\mathcal{J}_{(x_1,x_2+l)\times\R}(\tilde{u})\\
&\geq\mathcal{J}_{(x_1,x_2)\times\R}(u)-\mathcal{J}_{(x_1,x_2)\times\R}(\tilde{u})- C\frac{(\hat{h}(x_2)-h(x_2))^2}{l}>0,
\end{split}\]
in contradiction with the minimality of $u$. The proof is complete.
\end{proof}
\begin{remark}\label{TransPoss}
Later we consider a situation where the minimizer $u$ is defined in a bounded strips $[0,L]\times\R$ and satisfies a boundary condition of the form $u(L,\cdot)=\bar{u}_+(\cdot-\eta)$ where $\eta\in\R$ is a free parameter. In this situation the conclusion of Lemma \ref{NoMax} still applies with a simpler proof. Indeed the competing map $\tilde{u}$ can be defined exactly as in (\ref{compar1}) in the interval $[0,x_1]$ and as in (\ref{compar2}) in $(x_1,x_2)$ and, since $\eta\in\R$ can be chosen freely, by setting simply
\[\tilde{u}(x,\cdot)=u(x,\cdot-(\hat{h}(x_2)-h(x_2))),\;\;x\in[x_2,L].\]
\end{remark}
From Lemma \ref{NoMax} and Lemma \ref{to-baru+} it follows that, under the assumption  $\bar{u}_-=\bar{u}_+=\bar{u}$, if $u$ does not satisfy (\ref{unique}) then
\begin{equation}\label{Out}
\{x\in\R: q(x)>q^0\}\neq\emptyset.
\end{equation}
Indeed, since $\lim_{x\rightarrow\pm\infty}q(x)=0$ by Lemma \ref{to-baru+}, if $q(x)\leq q^0$, for all $x\in\R$, then Lemma \ref{NoMax} implies
\[q(x)\equiv 0,\]
and by (\ref{hprime})
\[h^\prime(x)\equiv 0,\]
and we conclude that (\ref{unique}) holds. We show that (\ref{Out}) cannot occur by constructing a map that competes energetically with $u$. Our construction is inspired by an argument from \cite{af2} (see Lemma 3.4).

\noindent
We fix a point $x^*\in\{x\in\R: q(x)>q^0\}$ and focus on the intervals $[\tilde{\xi}_1,\tilde{\xi}_2]\subset(\xi_1,\xi_2)$ defined by
\[\begin{split}
&\tilde{\xi}_1=\min\{x<x^*:q(x)\geq q^0\},\\
&\tilde{\xi}_2=\max\{x>x^*:q(x)\geq q^0\},
\end{split}\]
and
\[\begin{split}
&\xi_1=\max\{x<\tilde{\xi}_1:q(x)\leq\frac{q^0}{2}\},\\
&\xi_2=\min\{x>\tilde{\xi}_2:q(x)\leq\frac{q^0}{2}\}.
\end{split}\]
Note that
\begin{equation}\label{xtildex-interv}
\begin{split}
&q(\tilde{\xi_1})=q(\tilde{\xi_2})=q^0,\\
&q(\xi_1)=q(\xi_2)=\frac{q^0}{2},
\end{split}
\end{equation}
and also that
\begin{equation}\label{xtildex-interv1}
q(x)\in(\frac{q^0}{2},q^0),\quad\text{ for }\;\;x\in(\xi_1,\tilde{\xi}_1)\cup(\tilde{\xi}_2,\xi_2).
\end{equation}
We define the competing map $\tilde{u}$. We take
\[\tilde{u}(x,\cdot)=u(x,\cdot),\quad x\in(-\infty,\xi_1).\]
In $[\xi_1,\tilde{\xi}_1]\cup[\tilde{\xi}_2,\xi_2]$ we set
\[\tilde{u}(x,\cdot)=\hat{u}(x,\cdot),\]
with $\hat{q}=\hat{q}(x)$ and $\hat{h}=\hat{h}(x)$ defined as follows.
We set
\[\hat{q}(x)=q^0-q(x),\quad x\in[\xi_1,\tilde{\xi}_1]\cup[\tilde{\xi}_2,\xi_2].\]
Note that (\ref{xtildex-interv}) implies that $\hat{q}$ extends continuously $q$ at $x=\xi_1$ and $x=\xi_2$ and moreover that
\[\hat{q}(x)\in[0,\frac{q^0}{2}],\quad\text{ for }\;\;x\in[\xi_1,\tilde{\xi}_1]\cup[\tilde{\xi}_2,\xi_2].\]
In the interval $[\xi_1,\tilde{\xi}_1]$ we let $\hat{h}$ be the solution of (\ref{hat-hprime}) with initial condition $\hat{h}(\xi_1)=h(\xi_1)$. In the interval $[\tilde{\xi}_2,\xi_2]$ again we take the solution of (\ref{hat-hprime}) with initial condition $\hat{h}(\tilde{\xi}_2)=\hat{h}(\tilde{\xi}_1)$. It remains to specify $\tilde{u}(x,\cdot)$ for $x\in(\tilde{\xi}_1,\tilde{\xi}_2)\cup[\xi_2,+\infty)$. We take
\[
\begin{split}
&\tilde{u}(x,\cdot)=\bar{u}(\cdot-\hat{h}(\tilde{\xi}_1)),\quad x\in(\tilde{\xi}_1,\tilde{\xi}_2),\\
&\tilde{u}(x,\cdot)=u\Big(x,\cdot-(\hat{h}(\xi_2)-h(\xi_2))(1-\frac{x-\xi_2}{l})\Big),\quad x\in[\xi_2,\xi_2+l],\\
&\tilde{u}(x,\cdot)=u(x,\cdot),\quad x\in(\xi_2+l,+\infty).
\end{split}
\]
 With these definitions, one checks that $x\rightarrow\tilde{u}(x)$ is continuous and piece-wise smooth and coincides with $u(x,\cdot)$ outside $(\xi_1,\xi_2+l)$. Arguing as in the proof of Lemma \ref{NoMax} we show that
 \begin{equation}\label{jv-1111}
\mathcal{J}_{(\xi_2,\xi_2+l)\times\R}(\tilde{u})
-\mathcal{J}_{(\xi_2,\xi_2+l)\times\R}(u)\leq C\frac{(\hat{h}(\xi_2)-h(\xi_2))^2}{l},
\end{equation}
and
 \[
 \begin{split}
 &\mathcal{J}_{(\xi_1,\tilde{\xi}_1)\cup(\tilde{\xi}_2,\xi_2)\times\R}(\tilde{u})
 <\mathcal{J}_{(\xi_1,\tilde{\xi}_1)\cup(\tilde{\xi}_2,\xi_2)\times\R}(u),\\
 &c_0(\tilde{\xi}_2-\tilde{\xi}_1)
 =\mathcal{J}_{[\tilde{\xi}_1,\tilde{\xi}_2]\times\R}(\tilde{u})\leq\mathcal{J}_{[\tilde{\xi}_1,\tilde{\xi}_2]\times\R}(u).
 \end{split}
 \]
 Therefore, for $l>0$ large, we obtain
 \[\mathcal{J}_{(\xi_1,\xi_2+l)\times\R}(\tilde{u})
 <\mathcal{J}_{(\xi_1,\xi_2+l)\times\R}(u).\]
This contradicts the minimality of $u$ and concludes the proof of Theorem \ref{main}.

\section{The proof of Theorem \ref{main-0}}\label{DUE}
The proof is elementary, we sketch the argument.
\vskip.2cm
\noindent
Fix $a\in\{a_1,\ldots,a_N\}$ and, for $x_0\in\R$ and $l>1$, define $v_{x_0,l}:\R\rightarrow\R^m$ by setting
\begin{eqnarray*}
v_{x_0,l}(x):=\left\{\begin{array}{l}
u,\hskip1cm\text{ for }\;x\in(-\infty,x_0-l]\cup[x_0+l,+\infty),\\
a+(1-x+x_0-l)(u(x_0-l)-a),\\
\hskip1.2cm\;\text{ for }\;x\in(x_0-l,x_0-l+1],\\
a,\hskip1cm\text{ for }\;x\in(x_0-l+1,x_0+l-1),\\
a+(1+x-x_0-l)(u(x_0+l)-a),\\
\hskip1.2cm\;\text{ for }\;x\in[x_0+l-1,x_0+l).
\end{array}\right.
\end{eqnarray*}
Since $u$ and $u_x$ are bounded, there is a constant $\bar{J}>0$ such that
\begin{equation}\label{up}
{J}_{I_l(x_0)}(v_{x_0,l})<\bar{J},\;\text{ for all }\;x_0\in\R,\;l>1,
\end{equation}
where $I_l(x_0)=(x_0-l,x_0+l)$.
On the other hand, from (\ref{second-derW}), $\min_j\vert u(x)-a_j\vert\geq r$ for some $r\in(0,r_0]$ implies the existence of $w_r>0$ such that $W(u(x))>w_r$. Thus
\begin{equation}\label{down}
\min_j\vert u(x)-a_j\vert\geq r,\;\text{ for }\;x\in I_l(x_0)\;\;\Rightarrow\;\;{J}_{I_l(x_0)}(u)\geq 2l w_r.
\end{equation}
From (\ref{up}) and (\ref{down}) it follows that each ball $I_{l_r}(x_0)$ of radius $l_r=\bar{J}/{2 w_r}+1$ contains a point where $\min_j\vert u(x)-a_j\vert< r$.
Therefore if we consider the intervals $[2 k l_r, 2(k+1)l_r],\;\;k=0,1,\dots$ we have a sequence $\{x_k\},\;\; x_k\in(2 k l_r, 2(k+1)l_r)$ and a corresponding sequence $\{a_{j_k}\},\;a_{j_k}\in\{a_1,\dots,a_N\}$ with the property that $\vert u(x_k)-a_{j_k}\vert< r$. Since $W$ has a finite number of minima there is $a_+\in\{a_1,\dots,a_N\}$ and a subsequence $\{x_{k_i}\}$ such that
\[\vert u(x_{k_i})-a_+\vert\leq r,\;\;i=1,\dots.\]
This and the Cut-Off Lemma in \cite{af3}, if $r\in(0,\frac{r_0}{2}]$, imply
\[\vert u(x)-a_+\vert\leq r,\;\;\text{ for }\;\;x\geq x_{k_1}.\]
Since a similar statement hold for each $r\in(0,\frac{r_0}{2}]$ we conclude that
\[\lim_{x\rightarrow +\infty}u(x)=a_+\]
and a standard argument shows that actually the convergence is exponential.
The same argument shows that
\[\vert u(x)-a_-\vert\leq K e^{k x},\;\;x\leq 0,\]
for some $a_-\in\{a_1,\dots,a_N\}$.

If $a_-=a_+=a$ we have from $\lim_{x\rightarrow\pm\infty}u(x)=a$ and the definition of $v_{x_0,l}$ that $\lim_{l\rightarrow+\infty}{J}_{I_l(x_0)}(v_{x_0,l})=0$ while, in contradiction with the minimality of $u$, if $u\not\equiv a$ we have ${J}_{I_l(x_0)}(u)\geq\bar{J}^\prime$ for some $\bar{J}^\prime>0$ for all $l>1$ sufficiently large. This concludes the proof.

\section{A new approach to the existence of connections between global minima of the effective potential}\label{schatzman}

We develop in detail the approach sketched in the Introduction. We show in Lemma \ref{uleta-exists} and in Lemma \ref{ul-exists} that the minimizers in problems (\ref{P}) and (\ref{condition}) exist. In preparation to this we construct a map $\tilde{u}^{L,\eta}\in\mathcal{A}_{L,\eta}$ with finite energy and prove in Lemma \ref{near-a} that we can restrict to a subset of maps of $\mathcal{A}_{L,\eta}$ with a well controlled behavior for $y\rightarrow\pm\infty$.
\subsection{Existence of the minimizers $u^{L,\eta}$ and $u^L$.}\label{section31}
We start by showing that, in the minimization problem (\ref{P}), we can restrict to the subset of maps ${\rm{u}}\in\mathcal{A}_{L,\eta}$ that satisfy
\begin{equation}\label{uni-bound}
\|{\rm{u}}\|_{L^\infty(\mathcal{R}_L;\R^m)}\leq M,
\end{equation}
where $M>0$ is the constant in (\ref{M}).
 Indeed,
given ${\rm{u}}\in\mathcal{A}_{L,\eta}$, set ${\rm{u}}_M=0$ if ${\rm{u}}=0$ and ${\rm{u}}_M=\min\{\vert {\rm{u}}\vert, M\}{\rm{u}}/\vert {\rm{u}}\vert$ otherwise and note that (\ref{M}) implies
\[W({\rm{u}}_M)\leq W({\rm{u}}),\;\;\text{a.e.}\]
while we have
\[\vert\nabla {\rm{u}}_M\vert\leq\vert\nabla {\rm{u}}\vert,\;\;\text{a.e.}\]
 since the mapping ${\rm{u}}\rightarrow {\rm{u}}_M$ is a projection.
It follows
\[\mathcal{J}({\rm{u}})-\mathcal{J}({\rm{u}}_M)=\int_{\vert {\rm{u}}\vert\leq M}\Big(W({\rm{u}})-W({\rm{u}}_M)
+\frac{1}{2}(\vert\nabla {\rm{u}}\vert^2-\vert\nabla {\rm{u}}_M\vert^2)\Big)\geq 0.\]
that proves the claim.

We now show that $\mathcal{A}_{L,\eta}$ contains maps with finite energy.
 As before we set $c_0=J_\R(\bar{u}_\pm)$.

\begin{lemma}\label{B}
There exist $\tilde{u}^{L,\eta}\in\mathcal{A}_{L,\eta}$ and $C_0>0$ such that
\begin{equation}\label{B-est}
\mathcal{J}(\tilde{u}^{L,\eta})\leq C_0(1+\vert\eta\vert)+c_0L.
\end{equation}
\end{lemma}
\begin{proof}
For $L>1$ and $\eta\in\R$, define $\tilde{u}^{L,\eta}:[0,L]\times\R$ by setting
\begin{equation}\label{uleta-def}
\tilde{u}^{L,\eta}(x,y)=\left\{\begin{array}{l}
(1-x)\bar{u}_-(y)+x\bar{u}_+(y-\eta),\quad\text{ for }\;\;(x,y)\in[0,1]\times\R,\\
\bar{u}_+(y-\eta),\quad\text{ for }\;\;(x,y)\in(1,L]\times\R.
\end{array}\right.
\end{equation}
Since, as minimizers of (\ref{barmin}) $\bar{u}_\pm$ are also bounded by $M$, from (\ref{uleta-def}) we have
\begin{equation}\label{tildeu-M}
\|\tilde{u}^{L,\eta}\|_{L^\infty(\mathcal{R}_L;\R^m)}\leq M.
\end{equation}
We prove the estimate (\ref{B-est}) only for $\eta\geq 0$. The same argument applies to the case $\eta<0$.
From (\ref{uleta-def}), for $(x,y)\in[0,1]\times\R$, it follows
\begin{equation}\label{uleta-def-der}
\begin{split}
&\tilde{u}_x^{L,\eta}(x,y)=-
\bar{u}_-(y)+\bar{u}_+(y-\eta)\\
&=-
(\bar{u}_-(y)-a)+(\bar{u}_+(y-\eta)-a),\\
&\tilde{u}_y^{L,\eta}(x,y)=(1-x)\bar{u}_-^\prime(y)+x\bar{u}_+^\prime(y-\eta),
\end{split}
\end{equation}
and therefore, using also (\ref{n-bound}), we have
\begin{equation}\label{uleta-def-der0}
\begin{split}
&\vert\tilde{u}_x^{L,\eta}(x,y)\vert\leq 2M,\\
&\vert\tilde{u}_y^{L,\eta}(x,y)\vert\leq\bar{K},
\end{split}
\end{equation}
From (\ref{uleta-def-der0}) and (\ref{tildeu-M}) we obtain
\begin{equation}\label{in 0eta}
\mathcal{J}_{[0,1]\times[0,\eta]}(\tilde{u}_x^{L,\eta})
=\int_0^\eta\int_0^1\Big(W(\tilde{u}^{L,\eta})+\frac{1}{2}(\vert\tilde{u}_x^{L,\eta}\vert^2
+\vert\tilde{u}_y^{L,\eta}\vert^2)\Big){d} x{d} y\leq\eta C,
\end{equation}
for some $C>0$ independent of $L,\eta$.
From (\ref{uleta-def-der}) and (\ref{n-bound}) for $y\geq\eta$,
 we have
\begin{equation}\label{uleta-def-der1}
\begin{split}
&\int_0^1\vert\tilde{u}_x^{L,\eta}(x,y)\vert^2 {d} x\leq 2(\vert\bar{u}_-(y)-a\vert^2+\vert\bar{u}_+(y-\eta)-a\vert^2)\\
&\leq2{\bar{K}}^2(e^{-2\bar{k}y}+e^{-2\bar{k}(y-\eta)}),\\
&\int_0^1\vert\tilde{u}_y^{L,\eta}(x,y)\vert^2 dx\leq2\int_0^1((1-x)^2\vert\bar{u}_-^\prime(y)\vert^2+x^2\vert\bar{u}_+^\prime(y-\eta)\vert^2){d} x\\
&\leq \vert\bar{u}_-^\prime(y)\vert^2+\vert\bar{u}_+^\prime(y-\eta)\vert^2\leq {\bar{K}}^2(e^{-2\bar{k}y}+e^{-2\bar{k}(y-\eta)}).
\end{split}
\end{equation}
On the other hand, recalling also that $W(a)=0$, we have
\begin{equation}\label{uleta-pot}
\begin{split}
&W(\tilde{u}^{L,\eta}(x,y))=W\Big((1-x)\bar{u}_-(y)+x(\bar{u}_+(y-\eta)\Big)\\
&W\Big((1-x)(\bar{u}_-(y)-a)+x(\bar{u}_+(y-\eta)-a)+a\Big)\\
&\leq C^\prime\Big((1-x)\vert\bar{u}_-(y)-a\vert+x\vert\bar{u}_+(y-\eta)-a\vert\Big),
\end{split}
\end{equation}
with $C^\prime:=\max_{\vert z\vert\leq M}\vert W_u(a+z)\vert$.
It follows, for $y\geq\eta\geq 0$,
\begin{equation}\label{uleta-pot1}
\begin{split}
&\int_0^1W(\tilde{u}^{L,\eta}(x,y)){d} x\\
&\leq C^\prime\int_0^1\Big((1-x)\vert\bar{u}_-(y)-a\vert+x\vert\bar{u}_+(y-\eta)-a\vert\Big){d} x\\
&\leq C^\prime\bar{K}(e^{-\bar{k}y}+e^{-\bar{k}(y-\eta)}).
\end{split}
\end{equation}
 From (\ref{uleta-def-der1}) and (\ref{uleta-pot1}) we obtain
\begin{equation}\label{uleta-energy+}
\mathcal{J}_{[0,1]\times[\eta,+\infty)}(\tilde{u}^{L,\eta})\leq C_1,\quad\text{ for }\;\;\eta\geq 0
\end{equation}
where $C_1=C_1(\bar{k},\bar{K},C^\prime)$.
Similarly we have
\begin{equation}\label{uleta-energy-}
\mathcal{J}_{[0,1]\times(-\infty,0]}(\tilde{u}^{L,\eta})\leq C_1,\quad\text{ for }\;\;\eta\geq 0.
\end{equation}
These estimates, (\ref{in 0eta}) and the analogous estimates valid for $\eta<0$ imply
\[\mathcal{J}_{[0,L]\times\R}(\tilde{u}^{L,\eta})\leq 2C_1+C\vert\eta\vert+c_0(L-1)\leq C_0(1+\vert\eta\vert)+c_0L,\quad\text{ for }\;\;\eta\geq 0\]
with $C_0=\max\{2C_1,C\}$. The proof is complete.
\end{proof}
Lemma \ref{B} implies that, in the minimization problem (\ref{P}), we can restrict to the subset of $\mathcal{A}_{L,\eta}$ of the maps ${\rm{u}}$ that satisfy
\begin{equation}\label{less-energy-tilde}
\mathcal{J}({\rm{u}})\leq\mathcal{J}(\tilde{u}^{L,\eta}).
\end{equation}

\vskip.2cm
Next we show that we can further restrict $\mathcal{A}_{L,\eta}$ to the set of maps that converges uniformly to $a_\pm$ as $y\rightarrow\pm\infty$.

\begin{lemma}\label{near-a}
In the minimization problem (\ref{P}), the admissible set
$\mathcal{A}_{L,\eta}$ can be restricted to the subset of the maps ${\rm{u}}\in\mathcal{A}_{L,\eta}$ that satisfy (\ref{less-energy-tilde}), (\ref{uni-bound})
and
\begin{equation}\label{near-a1}
\begin{split}
&\vert {\rm{u}}(x,y)-a_+\vert\leq\frac{C_{L,\eta}}{\sqrt{y}},\quad\text{ for }\;\;y\geq y_{L,\eta},\\
&\vert {\rm{u}}(x,y)-a_-\vert\leq\frac{C_{L,\eta}}{\sqrt{-y}},\quad\text{ for }\;\;y\leq-y_{L,\eta}
\end{split}
\end{equation}
for some constants $C_{L,\eta}>0$, $y_{L,\eta}>0$.
\end{lemma}
\begin{proof}
From (\ref{n-bound}) we have
\begin{equation}\label{boundary-v-samll}\begin{split}
&\vert\bar{u}_-(y)-a_+\vert\leq\frac{r}{4},\quad\text{ for }\;\;y\geq y_r,\\
&\vert\bar{u}_+(y-\eta)-a_+\vert\leq\frac{r}{4},\quad\text{ for }\;\;y\geq y_r+\eta
\end{split}
\end{equation}
with $y_r=\frac{1}{\bar{k}}\ln{\frac{4\bar{K}}{r}}$.
Assume now $r\in(0,r_0]$, $r_0$ the constant in (\ref{second-derW}), and
define
\[Y_r:=\{y\geq y_r+\max\{0,\eta\}:\vert {\rm{u}}(x_y,y)-a_+\vert\geq\frac{r}{2},\;\;\text{ for some }\;\; x_y\in(0,L)\}.\]
Then, for $y\in Y_r$, we have
\[\vert {\rm{u}}(x_y,y)-\bar{u}_-(y)\vert\geq\vert {\rm{u}}(x_y,y)-a_+\vert-\vert\bar{u}_-(y)-a_+\vert\geq\frac{r}{4}.\]
It follows, recalling also the boundary condition ${\rm{u}}(0,\cdot)=\bar{u}_-$
\[\frac{r}{4}\leq\vert {\rm{u}}(x_y,y)-\bar{u}_-(y)\vert=\vert {\rm{u}}(x_y,y)-{\rm{u}}(0,y)\vert\leq L^\frac{1}{2}(\int_0^L\vert {\rm{u}}_x(x,y)\vert^2{d} x)^\frac{1}{2},\]
and therefore
\[\vert Y_r\vert\frac{r^2}{16 L}\leq\int_0^L\int_{Y_r}\vert{\rm{u}}_x(x,y)\vert^2{d} x{d} y\leq 2\mathcal{J}(\tilde{u}^{L,\eta})\]
that is
\begin{equation}\label{Y-measure}
\vert Y_r\vert\leq\frac{C_{L,\eta}^\prime}{r^2},\quad\text{ with }\;\;C_{L,\eta}^\prime=32L\mathcal{J}(\tilde{u}^{L,\eta}).
\end{equation}
It follows that there is an incresing sequence $y_{r,j},\;j=1,\ldots$ that diverges to $+\infty$ and is such that
\begin{equation}\label{yrj}
\begin{split}
&y_{r,1}\leq y_r+\max\{0,\eta\}+\frac{C_{L,\eta}^\prime}{r^2},\\
&y_{r,j}\in\R\setminus Y_r.
\end{split}
\end{equation}
This and (\ref{boundary-v-samll}) imply that
\[\vert {\rm{u}}(x,y)-a_+\vert\leq\frac{r}{2},\quad\text{ on }\;\;\partial R_j\]
where $R_j=(0,L)\times(y_{r,j},y_{r,j+1})$, $j=1,\ldots$

\noindent
We can then invoke the Cut-Off Lemma in \cite{af3} and conclude the existence of a map $\tilde{{\rm{u}}}$ that coincides with ${\rm{u}}$ for $y\leq y_{r,1}$ and satisfies
\begin{equation}\label{yrj-1}
\vert\tilde{{\rm{u}}}(x,y)-a_+\vert\leq\frac{r}{2},\quad\text{ for }\;\;x\in[0,L],\;y\geq y_{r,1}
\end{equation}
and
\[\mathcal{J}(\tilde{{\rm{u}}})\leq\mathcal{J}({\rm{u}}),\]
with strict inequality whenever $\vert Y_r\vert>0$. Therefore in the minimization problem (\ref{P}) we are allowed to suppose that ${\rm{u}}\in\mathcal{A}_{L,\eta}$ satisfies (\ref{yrj-1}).
By increasing the value of $C_{L,\eta}^\prime$ if necessary we can assume that
\[y_r+\max\{0,\eta\}\leq\frac{C_{L,\eta}^\prime}{r^2},\quad\text{ for }\;\;r\in(0,r_0].\]
Then $y=\frac{2C_{L,\eta}^\prime}{r^2}$ implies $y\geq y_{r,1}$ and therefore from the assumption that ${\rm{u}}$ satisfies (\ref{yrj-1}) it follows
\begin{equation}\label{yrj-2}
\vert {\rm{u}}(x,y)-a\vert\leq\sqrt{\frac{C_{L,\eta}^\prime}{2 y}},\quad\text{ for }\;\;y\geq\frac{2C_{L,\eta}^\prime}{r_0^2}.
\end{equation}
This proves (\ref{near-a1})$_1$ with $C_{L,\eta}=\sqrt{\frac{C_{L,\eta}^\prime}{2}}$ and $y_{L,\eta}=\frac{2C_{L,\eta}^\prime}{r_0^2}$. The other inequality is proved in a similar way.
\end{proof}
We are now in the position to prove the existence of the minimizers $u^{L,\eta}$ and $u^L$ of problems (\ref{P}) and (\ref{condition}).
\begin{lemma}\label{uleta-exists}
There exists $u^{L,\eta}\in\mathcal{A}_{L,\eta}$ that solves problem (\ref{P}):
\[\mathcal{J}(u^{L,\eta})=\min_{{\rm{u}}\in\mathcal{A}_{L,\eta}}\mathcal{J}({\rm{u}}).\]
Moreover $u^{L,\eta}$ satisfies (\ref{uni-bound}) and (\ref{near-a1}).
\end{lemma}
\begin{proof}
From Lemma \ref{B} we have
\begin{equation}\label{upper-lower}
0\leq\inf_{{\rm{u}}\in\mathcal{A}_{L,\eta}}\mathcal{J}({\rm{u}})\leq \mathcal{J}(\tilde{u}^{L,\eta})<+\infty.
\end{equation}
Let $\{{\rm{u}}_j\}_{j=1}^\infty\subset\mathcal{A}_{L,\eta}$ be a minimizing sequence. By Lemma \ref{near-a} and the discussion above we can assume that ${\rm{u}}_j$ satisfies (\ref{uni-bound}) and (\ref{near-a1}).
From (\ref{upper-lower}) we have
\[ \int_{\mathcal{R}_L}\frac{1}{2}\vert\nabla {\rm{u}}_j\vert^2 {d} x{d} y\leq\mathcal{J}({\rm{u}}_j)\leq\mathcal{J}
(\tilde{u}^{L,\eta}). \]
Hence, using also that $\|{\rm{u}}_j\|_{L^\infty(\mathcal{R}_L;\R^m)}\leq M$, by weak compactness we have that, possibly by passing to a subsequence,
\[  {\rm{u}}_j\rightharpoonup u, \text{ in } W^{1,2}_{\mathrm{loc}}(\mathcal{R}_L;\R^m),\]
for some $u\in W^{1,2}_{\mathrm{loc}}(\mathcal{R}_L;\R^m)$.

By compactness of the embedding we can assume that ${\rm{u}}_j\to u$ strongly in $L^{2}_{\text{loc}}(\mathcal{R}_L;\R^m)$ and therefore, along a further subsequence,
\begin{equation}\label{pointwise}
\lim_{j\to+\infty}{\rm{u}}_j(x,y)=u(x,y), \text{ a.e.\ in }\mathcal{R}_L.
\end{equation}
Weak lower semi-continuity of the $L^2$ norm gives
\[ \liminf_{j\to+\infty}\int_{\mathcal{R}_L}\frac{1}{2}\vert\nabla {\rm{u}}_j\vert^2 {d} x{d} y\geq\int_{\mathcal{R}_L}\frac{1}{2}\vert\nabla u\vert^2 {d} x{d} y,
\]
and by Fatou's lemma,
\[ \liminf_{j\to+\infty}\int_{\mathcal{R}_L} W({\rm{u}}_j) {d} x{d} y\geq\int_{\mathcal{R}_L} W(u) {d} x{d} y.\]
Moreover from (\ref{pointwise}) we have that $u$ satisfies (\ref{uni-bound}) and (\ref{near-a1}). It follows that we can identify the map $u$ with the sought minimizer $u^{L,\eta}$.
The proof is complete.
\end{proof}

The minimizer $u^{L,\eta}$ given by Lemma \ref{uleta-exists} satisfies the bound \[\|u^{L,\eta}\|_{L^\infty(\mathcal{R}_L;\R^m)}\leq M\]
 independently of $L,\eta$ therefore the smoothness of $W$ and of the boundary conditions $\bar{u}_-$ and $\bar{u}_+(\cdot-\eta)$ implies via elliptic regularity that
\begin{equation}\label{smoothness}
\|u^{L,\eta}\|_{C^{2,\gamma}(\mathcal{R}_L;\R^m)}\leq C,
\end{equation}
for some constant $C>0$ and $\gamma\in(0,1)$ independent of $L,\eta$.
\begin{lemma}\label{ul-exists}
There exists $\bar{\eta}\in\R$ and $u^L\in\mathcal{A}_{L,\bar{\eta}}$ such that
\[\mathcal{J}(u^L)=\min_\eta\mathcal{J}(u^{L,\eta})\leq C_0+c_0L.\]
\end{lemma}
\begin{proof}

1. There exists $\bar{y}>0$ such that
\begin{equation}\label{energ-low-b}
(\vert\eta\vert-2\bar{y})\frac{\vert a\vert^2}{L}\leq\mathcal{J}(\mathrm{u}),\quad\text{ for }\;\;
\vert\eta\vert\geq 2\bar{y},\;\mathrm{u}\in\mathcal{A}_{L,\eta},
\end{equation}
where $a=\frac{a_+-a_-}{2}$.

\noindent
Assume first $\eta\geq 0$. Since both $\bar{u}_-$ and $\bar{u}_+$ satisfy (\ref{n-bound}) there exists $\bar{y}>0$ such that
\[
\begin{split}
&\vert\bar{u}_-(y)-a_+\vert\leq\frac{1}{2}\vert a\vert,\quad\text{ for }\;\;y\geq\bar{y},\\
&\vert\bar{u}_+(y-\eta)-a_-\vert\leq\frac{1}{2}\vert a\vert,\quad\text{ for }\;\;y\leq\eta-\bar{y}.
\end{split}
\]
It follows
\[\begin{split}
&\vert\bar{u}_+(y-\eta)-\bar{u}_-(y)\vert\geq 2\vert a\vert-\vert\bar{u}_+(y-\eta)-a_--(\bar{u}_-(y)-a_+)\vert\\
&\geq\vert a\vert,\;\;\text{for}\;\;\bar{y}\leq y\leq\eta-\bar{y}
\end{split}\]
and therefore, since ${\rm{u}}\in\mathcal{A}_{L,\eta}$ implies
\[{\rm{u}}(L,y)-{\rm{u}}(0,y)=\bar{u}_+(y-\eta)-\bar{u}_-(y),\]
we have
\[\vert a\vert\leq\vert{\rm{u}}(L,y)-{\rm{u}}(0,y)\vert\leq L^\frac{1}{2}(\int_0^L\vert {\rm{u}}_x(x,y)\vert^2{d} x)^\frac{1}{2},\quad\text{ a.e. in }(\bar{y},\eta-\bar{y})\]
and in turn
\[(\eta-2\bar{y})\vert a\vert^2\leq L\int_{\bar{y}}^{\eta-\bar{y}}\int_0^L\vert {\rm{u}}_x(x,y)\vert^2{d} x\leq L \mathcal{J}({\rm{u}}).\]
This and similar estimates valid for $\eta<0$ prove 1.

2. Let $u^{L,\eta_j},\;j=1\ldots$ a minimizing sequence. From Lemma \ref{B} we can choose a minimizing sequence that satisfies
\begin{equation}\label{mini-sequence}
\lim_{j\rightarrow+\infty}\mathcal{J}(u^{L,\eta_j})=\inf_\eta\mathcal{J}(u^{L,\eta})\leq C_0+c_0L
\end{equation}
From (\ref{smoothness}), by passing to a subsequence if necessary, we can assume that there is a continuous function $u^L$ such that
\[\lim_{j\rightarrow+\infty}u^{L,\eta_j}=u^L\]
uniformly in compact sets. From (\ref{mini-sequence}) and (\ref{energ-low-b}) it follows that the sequence $\eta_j$, $j=1,\ldots$ is bounded and therefore, along a further subsequence,
\[\lim_{j\rightarrow+\infty}\eta_j=\bar{\eta}.\]
This and the uniform convergence of $u^{L,\eta_j}$ to $u^L$ imply that $u^L$ satisfies the boundary conditions in
 $\mathcal{A}_{L,\bar{\eta}}$. From this point we can proceed as in Lemma \ref{uleta-exists} to conclude that
  $u^L$ is the sought minimizer. The proof is complete
\end{proof}
The minimizer $u^L$ determined in Lemma \ref{ul-exists} can be identified with $u^{L,\bar{\eta}}$. Indeed from the fact that $u^L$ satisfies the boundary conditions for $\eta=\bar{\eta}$ we have
\[\mathcal{J}(u^{L,\bar{\eta}})\leq\mathcal{J}(u^L)=\min_\eta\mathcal{J}(u^{L,\eta})\leq\mathcal{J}(u^{L,\bar{\eta}}).\]
In the following when it is clear from the context we simply write $u$ instead of $u^L$ and we do the same with other functions of $L$ that we introduce later.
\subsection{Basic Lemmas}\label{section32}

In this section we prove a few lemmas that are basic for deriving estimates on $u^L$ that are uniform in $L$ and  allow to pass to the limit in (\ref{uS-lim}). In Lemma \ref{first-exp-decay} we prove that $u^L$ decays exponentially to $a_\pm$ as $y\rightarrow\pm\infty$. In Lemma \ref{ul-bounds} we show that $\int_0^L\int_\R\vert u_x^L\vert^2{d} x{d} y$ is uniformly bounded in $L$. This is a simple result which is essential for the analysis that follows.
\vskip.2cm
Note that $u^L$ satisfies (\ref{smoothness}) and is a classical solution of (\ref{system}). Note also that,
since $\bar{\eta}$ in Lemma \ref{ul-exists} depends only on $L$, when applied to $u=u^L$, (\ref{near-a1}) takes the form
\begin{equation}\label{near-a-2}
\begin{split}
&\vert u(x,y)-a_+\vert\leq\frac{C_L}{\sqrt{y}},\quad\text{ for }\;\;y\geq y_L,\\
&\vert u(x,y)-a_-\vert\leq\frac{C_L}{\sqrt{-y}},\quad\text{ for }\;\;y\leq-y_L
\end{split}
\end{equation}
for some constants $C_L>0$, $y_L>0$.

The fact that $u^L$ solves (\ref{system}) implies a sharper asymptotic behavior for $y\rightarrow\pm\infty$.
\begin{lemma}\label{first-exp-decay}
 There exist constants $k, K>0$ independent of $L>0$ and such that $u=u^L$ satisfies
\begin{equation}\label{exp-decay-2}
\begin{split}
&\vert u(x,y)-a_+\vert\leq r_0 e^{-k(y-y_L)},\quad\text{ for }\;\;y\geq y_L,\\
&\vert u(x,y)-a_-\vert\leq r_0 e^{k(y_L+y)},\quad\text{ for }\;\;y\leq-y_L,
\end{split}
\end{equation}
where $r_0>0$ is the constant in (\ref{second-derW}).

Moreover, for $\alpha\in\R^2$, $1\leq\vert\alpha\vert\leq 2$, it results
\begin{equation}\label{exp-decay2}
\begin{split}
&\vert (D^\alpha{u})(x,y)\vert\leq K e^{-k(y-y_L)},\;\;\text{ for }\;\; y\geq y_L,\\
&\vert (D^\alpha{u})(x,y)\vert\leq K e^{k(y+y_L)},\;\;\text{ for }\;\; y\leq -y_L.
\end{split}
\end{equation}
 \end{lemma}
\begin{proof} A standard computation yields that, provided $\vert u-a_+\vert\leq r_0$, we have
\begin{equation}\label{standard}
W_u(u)\cdot(u-a_+)=(W_u(u)-W_u(a_+))\cdot(u-a_+)\geq\gamma^2\vert u-a_+\vert^2,
\end{equation}
where $\gamma$ is the constant in (\ref{second-derW}) and we have used $W_u(a_+)=0$. On the other hand, since $u=u^L$ solves (\ref{system}) we have
\[\Delta\vert u-a_+\vert^2\geq 2(\Delta u)\cdot(u-a_+)=2W_u(u)\cdot(u-a_+),\]
and therefore (\ref{standard}) implies
\begin{equation}\label{delta-Q}
\Delta\vert u-a_+\vert^2\geq 2\gamma^2\vert u-a_+\vert^2\quad\text{ on }\;\;\mathcal{R}_L\cap\{\vert u-a_+\vert\leq r_0\}.
\end{equation}
From (\ref{near-a-2}), by redefining $y_L$ if necessary, we can assume
\begin{equation}\label{yL-yL+l}
\vert u(x,y)-a_+\vert\leq r_0\quad\text{ for }\;\;x\in[0,L],\;y\geq y_L.
\end{equation}
Similarly from (\ref{n-bound}) and the fact that $\bar{\eta}$ depends only on $L$ we can assume
\begin{equation}\label{0Lboundary}
\begin{split}
&\vert u(0,y)-a_+\vert=\vert\bar{u}_-(y)-a_+\vert\leq r_0e^{-\bar{k}(y-y_L)},\quad\text{ for }\;\;y\geq y_L,\\
&\vert u(L,y)-a_+\vert=\vert\bar{u}_+(y-\bar{\eta})-a_+\vert\leq r_0e^{-\bar{k}(y-y_L)},\quad\text{ for }\;\;y\geq y_L.
\end{split}
\end{equation}

Let $\tilde{\gamma}>0$ be a number to be chosen later and, for $l>0$, set
\[\varphi_l(x,y):= r_0^2\frac{\cosh{\tilde{\gamma}(l-(y-y_L))}}{\cosh{\tilde{\gamma} l}},\quad\text{ for }\;\;x\in[0,L],\;y\in[y_L,y_L+2l].\]
Observe that we have
\[\begin{split}
&\varphi_l(x,y)=r_0^2\frac{1+e^{-2\tilde{\gamma} l}e^{2\tilde{\gamma}(y-y_L)}}
{1+e^{-2\tilde{\gamma} l}}e^{-\tilde{\gamma}(y-y_L)}\\
&\geq r_0^2e^{-\tilde{\gamma}(y-y_L)},\quad\text{ for }\;\;x\in[0,L],\;y\in[y_L,y_L+2l].
\end{split}\]
and
\begin{equation}\label{phi-upper-bound}
2r_0^2e^{-\tilde{\gamma}(y-y_L)}\geq\varphi_l(x,y),\quad\text{ for }\;\;x\in[0,L],\;y\in[y_L,y_L+l]
\end{equation}
and therefore that from (\ref{0Lboundary}), provided we assume $\tilde{\gamma}\leq 2\bar{k}$, it results
\begin{equation}\label{lateral}
\begin{split}
&\vert u(0,y)-a_+\vert^2\leq\varphi_l(0,y),\quad\text{ for }\;\;y\in[y_L,y_L+2l],\\
&\vert u(L,y)-a_+\vert^2\leq\varphi_l(L,y),\quad\text{ for }\;\;y\in[y_L,y_L+2l].
\end{split}
\end{equation}
 This, (\ref{yL-yL+l}) and $\varphi_l(x,y_L)=\varphi_l(x,y_L+2l)=r_0^2$ imply
\begin{equation}\label{all-boundary}
\vert u(x,y)-a_+\vert^2\leq\varphi_l(x,y),\quad\text{ on }\;\;\partial R_l
\end{equation}
where $R_l:=(0,L)\times(y_L,y_L+2l)$.
Now note that $\varphi_l$ is a solution of
\[\Delta\varphi=\tilde{\gamma}^2\varphi\]
and assume $\tilde{\gamma}\leq\min\{\sqrt{2}\gamma,2\bar{k}\}$. Then from (\ref{delta-Q}), (\ref{all-boundary}) and the maximum principle we conclude
\begin{equation}\label{insideRl}
\vert u(x,y)-a_+\vert^2\leq\varphi_l(x,y),\quad\text{ on }\;\; R_l,
\end{equation}
and therefore from (\ref{phi-upper-bound}) we have
\[\vert u(x,y)-a_+\vert^2\leq 2r_0^2e^{-\tilde{\gamma}(y-y_L)},\quad\text{ for }\;\;y\in[y_L,y_L+l].\]
Since this is valid for all $l>0$ we obtain
\[\vert u(x,y)-a_+\vert\leq\sqrt{2}r_0e^{-\frac{\tilde{\gamma}}{2}(y-y_L)},\quad\text{ for }\;\;y\geq y_L\]
which implies (\ref{exp-decay-2})$_1$ with $k=\frac{\tilde{\gamma}}{2}$ after changing $y_L$ to $y_L+\frac{1}{\tilde{\gamma}}\ln{2} $. The other case is discussed in a similar way. Once (\ref{exp-decay-2}) is established the estimates for the derivatives follows from the smoothness of $W$ and of the boundary conditions and elliptic regularity. The proof is complete.
\end{proof}
From Lemma \ref{first-exp-decay} and (\ref{smoothness}) it follows that, for $u=u^L$, the quantities
\begin{equation}\label{quantities}
\begin{split}
&q(x)=\min_{r\in\R,\mathrm{p}\in\{-,+\}}\|u(x,\cdot)-\bar{u}_{\mathrm{p}}(\cdot-r)\|,\\
&\min_{r\in\R,\mathrm{p}\in\{-,+\}}\|u(x,\cdot)-\bar{u}_{\mathrm{p}}(\cdot-r)\|_1,\\
&\int_\R W(u(x,y)){d} y,\quad\int_\R\vert u_x(x,y)\vert^2{d} y,\quad\int_\R\vert u_y(x,y)\vert^2{d} y,
\end{split}
\end{equation}
are well defined and continuous for $x\in[0,L]$.
\begin{lemma}\label{ul-bounds}
Let $u^L$ be as in Lemma \ref{ul-exists}. Then
\begin{equation}
\begin{split}
&c_0L\leq\mathcal{J}(u^L)\leq c_0L+C_0,\\
&\int_0^L\int_\R\vert u^L_x\vert^2{d} x{d} y\leq 2C_0.
\end{split}
\end{equation}
where $c_0$ and $C_0$ are the constants in Lemma \ref{B}.
\end{lemma}
\begin{proof}
Lemma \ref{near-a} implies $\lim_{y\rightarrow\pm\infty}u^L=a_\pm$ and therefore from $c_0=J_\R(\pm\bar{u}_\pm)$  and the minimizing character of $\bar{u}_\pm$ it follows
\begin{equation}\label{min-char}
\int_\R\Big(W(u^L(x,y)+\frac{1}{2}\vert u^L_y(x,y)\vert^2\Big){d} y\geq c_0,\quad\text{ for }\;x\in[0,L].
\end{equation}
From this and Lemma \ref{ul-exists} it follows
\[\begin{split}
&\frac{1}{2}\int_0^L\int_\R\vert u^L_x\vert^2{d} x{d} y\\
&\leq\int_0^L\Big(\int_\R\Big(W(u^L)+\frac{1}{2}\vert u^L_y\vert^2\Big){d} y-c_0\Big){d} x+\frac{1}{2}\int_0^L\int_\R\vert u^L_x\vert^2{d} x{d} y\leq C_0.
\end{split}\]
The proof is complete
\end{proof}

\begin{lemma}\label{yvar}
Assume $u=u^L$ is as in Lemma \ref{ul-exists}. Then $u$ satisfies (\ref{hamilton}) and (\ref{uxuyconst}) for $x\in[0.L]$ and for some constants $\omega=\omega_L$ and $\tilde{\omega}=\tilde{\omega}_L$.
Moreover
\begin{equation}\label{uxuyconst-1}
\begin{split}
&0\leq\omega\leq\frac{C_0}{L},\\
&\tilde{\omega}=0.
\end{split}
\end{equation}
\end{lemma}
\begin{proof}
The first part of the Lemma is proved as in  Lemma \ref{properties}. From (\ref{hamilton}) and (\ref{min-char}) we have
\begin{equation}\label{+omega}
\frac{1}{2}\int_\R\vert u_x(x,y)\vert^2{d} y=J_\R(u(x,\cdot))-c_0+\omega\geq\omega,\quad\text{ for }\;\;x\in[0,L].
\end{equation}
Recalling the boundary value $u(0,\cdot)=\bar{u}_-$ that implies $J_\R(u(0,\cdot))=c_0$ we see that, for $x=0$,  (\ref{+omega}) implies
\begin{equation}\label{+omega1}
\frac{1}{2}\int_\R\vert u_x(0,y)\vert^2{d} y=\omega
\end{equation}
and therefore that $\omega\geq 0$. By integrating (\ref{+omega}) on $[0,L]$ and using also Lemma \ref{ul-bounds} yields
\begin{equation}\label{+omega11}
\omega L\leq\int_0^L\frac{1}{2}\int_\R\vert u_x(x,y)\vert^2{d} x{d} y\leq C_0
\end{equation}
that is $\omega\leq\frac{C_0}{L}$.
 To prove that $\tilde{\omega}=0$ it suffices to observe that $u=u^L$ is a solution of (\ref{condition}) and therefore the restriction (\ref{0boundary}) in the proof of Lemma \ref{properties} can be removed.
The proof is complete.
\end{proof}

\subsection{Structural properties of $u^L$}\label{section33}
We are now able to derive detailed information on the structure of the minimizer $u^L\in\mathcal{A}_{L,\bar{\eta}}$ determined in Lemma \ref{ul-exists}. This knowledge of $u^L$, in particular the fact that, as we show below,  $\bar{\eta}$ is bounded independently of $L$, will allow us to pass to the limit in (\ref{uS-lim}) and show that the limit map is a solution of (\ref{system}) with the properties required in Theorem \ref{scha-th}.

Let $u=u^L$. Fix ${p}\in(0,q^0]$ and let $S_{p}\subset[0,L]$ be the complement of the set $\tilde{S}_{p}$ defined by
\begin{equation}\label{size-cond}
\tilde{S}_{p}:=\{x\in(0,L):\|u_x(x,\cdot)\|^2>e_{p}\}
\end{equation}
where $e_{p}$ is the constant in Lemma \ref{away}.
From Lemma \ref{ul-bounds}
 it follows that the measure of $\tilde{S}_{p}$ is bounded independently of $L>1$. Indeed we have:
\[\vert\tilde{S}_{p}\vert\frac{e_{p}}{2}\leq\frac{1}{2}\int_0^L \|u_x(x,\cdot)\|^2{d} x\leq C_0\]
and therefore
\begin{equation}\label{esse-tilde}
\vert\tilde{S}_{p}\vert\leq\frac{2C_0}{e_{p}},\quad\text{ and }\quad\vert S_{p}\vert\geq L-\frac{2C_0}{e_{p}}.
\end{equation}
From Lemma \ref{yvar} we know that $u=u^L$ satisfies (\ref{hamilton}) with $\omega\geq 0$ and it follows
\begin{equation}\label{EffPotKin}
J(u(x,\cdot))-c_0\leq\frac{1}{2}\|u_x(x,\cdot)\|^2\leq\frac{e_{p}}{2},\;\;x\in S_{p}.
\end{equation}
This and Lemma \ref{away} imply
\begin{equation}\label{pLessq}
\begin{split}
& q(x)=\min_{\mathrm{p}\in\{-,+\}}\min_{r\in\R}\|u(x,\cdot)-\bar{u}_{\mathrm{p}}(\cdot-r)\|\\
&\leq\min_{\mathrm{p}\in\{-,+\}}\min_{r\in\R}\|u(x,\cdot)-\bar{u}_{\mathrm{p}}(\cdot-r)\|_1\leq{p}\leq q^0.
\end{split}
\end{equation}
 Therefore from Lemma \ref{lemmaw} we can associate to each $x\in S_{p}$ unique $\bar{u}\in\{\bar{u}_-,\bar{u}_+\}$
and $h(x)\in\R$ that allow to represent $u$ as in (\ref{decomp-i1}):
 \[u(x,\cdot)=\bar{u}(\cdot-h(x))+v(x,\cdot-h(x)),\;\;x\in S_{p}\]
 with
 \begin{equation}\label{vL}
 v(x,y)=v^L(x,y):=u^L(x,y+h(x))-\bar{u}(y)
 \end{equation} that satisfies
\[\langle v(x,\cdot),\bar{u}^\prime\rangle=0.\]
If needed, we indicate that the map $\bar{u}\in\{\bar{u}_-,\bar{u}_+\}$ associated to $x\in S_{p}$ depends on $x$ by using the notation $\bar{u}=\bar{u}_{\mathrm{p}_x}$ with $\mathrm{p}_x\in\{-,+\}$. Note that $\mathrm{p}_x=Const$ in each interval contained in $S_{p}$.
and from (\ref{pLessq}) we have
\begin{equation}\label{pLessq1}
\begin{split}
&\|v(x,\cdot)\|=\|u(x,\cdot)-\bar{u}(\cdot-h(x))\|\\
&=\min_{\mathrm{p}\in\{-,+\}}\min_{r\in\R}\|u(x,\cdot)-\bar{u}_{\mathrm{p}}(\cdot-r)\|
\leq{p}\leq q^0.
\end{split}
\end{equation}
Note that (\ref{w12-diffe}) in Lemma \ref{lemmaw1} implies
\begin{equation}\label{vy}
\|v_y(x,\cdot)\|\leq\|v(x,\cdot)\|_1\leq\bar{C}p^\frac{1}{2}.
\end{equation}
From (\ref{exp-decay-2}) and (\ref{exp-decay2}) and Lemma \ref{hprime-hpprime?} it follows that
the map $S_{p}\ni x\rightarrow h(x)\in\R$ is continuously differentiable and
\[\vert h(x)\vert, \vert h^\prime(x)\vert\leq C_L,\]
for some constant $C_L>0$ that may depend on $L>1$. Moreover from Lemma \ref{represent}, and Lemma \ref{yvar} that yields $\tilde{\omega}=0$, we have that the expression (\ref{hprime}) and (\ref{xenergy}) respectively are valid for $h^\prime(x)$ and $\|u_x(x,\cdot)\|^2$ for $x\in S_{p}$.
We assume $p\leq\frac{2^\frac{1}{2}-1}{\bar{C}}\|\bar{u}^\prime\|$ then we have
\[\frac{\|v_y(x,\cdot)\|^2}{\|\bar{u}^\prime+v_y(x,\cdot)\|^2}
\leq\frac{\bar{C}^2p}{(\|\bar{u}^\prime\|-\bar{C}p^\frac{1}{2})^2}\leq\frac{1}{2}\]
and (\ref{xenergy}) implies
\begin{equation}\label{vxW}
\|v_x(x,\cdot)\|\leq 2^\frac{1}{2}\|u_x(x,\cdot)\|,\;\;x\in S_{p}.
\end{equation}
We also observe that, on the basis of (\ref{pLessq1}), we can use (\ref{W2geqq2-i}) and deduce from (\ref{EffPotKin}) that
\[\frac{1}{2}\mu\|v_y(x,\cdot)\|^2\leq\mathcal{W}(v)=J(u(x,\cdot))-c_0\leq\frac{1}{2}\|u_x(x,\cdot)\|^2\]
and in turn
\begin{equation}\label{vyW}
\|v_y(x,\cdot)\|\leq\frac{1}{\mu^\frac{1}{2}}\|u_x(x,\cdot)\|,\;\;x\in S_{p}.
\end{equation}

 \begin{proposition}\label{eta-bounded}
Let $u=u^L$ the minimizer in Lemma \ref{ul-exists}. Then there is a constant $C>0$ independent of $L$ such that
\begin{equation}\label{eta-bound1}
\vert\bar{\eta}\vert\leq C.
\end{equation}
\end{proposition}
\begin{proof}
 $\tilde{S}_{p}$ is the union of a countable family of intervals $\tilde{S}_{p}=\cup_j(\alpha_j,\beta_j)$. Therefore we have
\begin{equation}\label{etabar-exp}
\vert\bar{\eta}\vert\leq\int_{S_{p}}\vert h^\prime\vert{d} x+\sum_j\vert h(\beta_j)-h(\alpha_j)\vert.
\end{equation}
Let $\lambda>0$ a small number to the chosen later and set $I_\lambda=\{j:\beta_j-\alpha_j\leq\lambda\}$, $\tilde{I}_\lambda=\{j:\beta_j-\alpha_j>\lambda\}$. Note that $\tilde{I}_\lambda$ contains at most $\frac{\vert\tilde{S}_{p}\vert}{\lambda}\leq\frac{2C_0}{\lambda e_{p}}$ elements. For $j\in I_\lambda$ and $\xi\in(\alpha_j,\beta_j)$ we have
\[
\vert u(\xi,y)-u(\alpha_j,y)\vert\leq\int_{\alpha_j}^\xi\vert u_x(x,y)\vert {d} x\leq\vert\xi-\alpha_j\vert^\frac{1}{2}(\int_{\alpha_j}^\xi\vert u_x(x,y)\vert^2 {d} x)^\frac{1}{2},\]
and therefore
\[\int_\R\vert u(\xi,y)-u(\alpha_j,y)\vert^2{d} y\leq\vert\beta_j-\alpha_j\vert\int_{\alpha_j}^{\beta_j}\int_\R\vert u_x(x,y)\vert^2{d} y{d} x\leq\lambda C_0
\]
where $C_0$ is the constant in Lemma \ref{ul-bounds}.
From this estimate and (\ref{pLessq1}) that implies
\[\|u(\alpha_j,\cdot)-\bar{u}(\cdot-h(\alpha_j))\|\leq {p}\leq q^0\]
 it follows that, if ${p}$ and $\lambda$ are sufficiently small, then, for each $x\in(\alpha_j,\beta_j)$, $u(x,\cdot)$ satisfies the conditions in Lemma \ref{lemmaw} ensuring that $\bar{u}$ and $h(x)$ are uniquely determined and either $\bar{u}=\bar{u}_-$ or $\bar{u}=\bar{u}_+$ for every $x\in[\alpha_j,\beta_j]$. Moreover ${h}$ is a smooth function of $u(x,\cdot)$ and, from (\ref{derivative}),  we have
\[{h}^\prime(x)=\frac{\langle u_x(x,\cdot),\bar{u}^\prime(\cdot-{h}(x))\rangle}{\|\bar{u}^\prime\|^2+\langle u(x,\cdot)-\bar{u}(\cdot-{h}(x)),\bar{u}^{\prime\prime}(\cdot-{h}(x))\rangle},\]
which implies
\begin{equation*}
\vert{h}^\prime(x)\vert\leq C\|u_x(x,\cdot)\|,\quad\text{ for }\;\;x\in[\alpha_j,\beta_j],\;j\in I_\lambda.
\end{equation*}

\noindent
Therefore we have
\begin{equation}\label{first-s}
\begin{split}
&\sum_{j\in I_\lambda}\vert h(\beta_j)-h(\alpha_j)\vert\leq
\int_{\cup_{j\in I_\lambda}[\alpha_j,\beta_j]}\vert{h}^\prime(x)\vert {d} x\\
&\leq C\int_{\cup_{j\in I_\lambda}[\alpha_j,\beta_j]}\|u_x\| {d} x\leq C\vert\tilde{S}_{p}\vert^\frac{1}{2}(\int_0^L\|u_x\|^2 {d} x)^\frac{1}{2}\leq (2C_0)^\frac{1}{2}C\vert\tilde{S}_{p}\vert^\frac{1}{2}.
\end{split}
\end{equation}
Assume now $j\in\tilde{I}_\lambda$ and observe that, (\ref{n-bound}) implies  that there is $\bar{y}>0$ such that, for $\mathrm{p},\mathrm{q}\in\{-,+\}$
\begin{equation}\label{elem-observ}
\begin{split}
&\vert\bar{u}_\mathrm{p}(y)-\bar{u}_\mathrm{q}(y-r)\vert\geq \vert a\vert,\;\;\text{ for }
\bar{y}\leq y\leq r-\bar{y},\;\text{ if }\;r\geq 2\bar{y},\\
&\vert\bar{u}_\mathrm{p}(y)-\bar{u}_\mathrm{q}(y-r)\vert\geq \vert a\vert,\;\;\text{ for }
 r+\bar{y}\leq y\leq-\bar{y},\;\text{ if }\;r\leq-2\bar{y},
\end{split}
\end{equation}
where as before $a=\frac{a_+-a_-}{2}$.

Consider first the indices $j\in\tilde{I}_\lambda$ such that $\vert h(\beta_j)-h(\alpha_j)\vert\leq 4\bar{y}$. We have
\begin{equation}\label{secon-s}
\sum_{j\in\tilde{I}_\lambda,\vert h(\beta_j)-h(\alpha_j)\vert\leq4\bar{y}}\vert h(\beta_j)-h(\alpha_j)\vert\leq4\bar{y}
\frac{\vert\tilde{S}_{p}\vert}{\lambda}.
\end{equation}
If $r>4\bar{y}$ the interval $(\bar{y},r-\bar{y})$ (if $r<-4\bar{y}$ the interval $(r+\bar{y},-\bar{y})$) has measure larger then $\frac{\vert r\vert}{2}$.
Therefore, for each $j\in\tilde{I}_\lambda$ with $\vert h(\beta_j)-h(\alpha_j)\vert>4\bar{y}$, there are $y_j^0,y_j^1$, $ y_j^1-y_j^0=\vert h(\beta_j)-h(\alpha_j)\vert/2$ such that
\begin{equation}\label{bigger-gamma}
\begin{split}
&\vert u(\beta_j,y)-u(\alpha_j,y)\vert
\geq\vert\bar{u}_{\mathrm{p}_{\beta_j}}(y-h(\beta_j))-\bar{u}_{\mathrm{p}_{\alpha_j}}(y-h(\alpha_j))\vert\\
&-\vert u(\beta_j,y)-\bar{u}_{\mathrm{p}_{\beta_j}}(y-h(\beta_j))\vert
-\vert u(\alpha_j,y)-\bar{u}_{\mathrm{p}_{\alpha_j}}(y-h(\alpha_j))\vert\\
&=\vert\bar{u}_{\mathrm{p}_{\beta_j}}(y-h(\beta_j))-\bar{u}_{\mathrm{p}_{\alpha_j}}(y-h(\alpha_j))\vert\\
&-\vert v(\beta_j,y-h(\beta_j))\vert
-\vert v(\alpha_j,y-h(\alpha_j))\vert\\
&\geq\vert\bar{u}_{\mathrm{p}_{\beta_j}}(y-h(\beta_j))-\bar{u}_{\mathrm{p}_{\alpha_j}}(y-h(\alpha_j))\vert-2(2\bar{C})^\frac{1}{2}p^\frac{3}{4}\\
&\geq\vert a\vert-2(2\bar{C})^\frac{1}{2}p^\frac{3}{4}\geq\frac{\vert a\vert}{2},\quad\text{ for }\;\;y\in(y_j^0,y_j^1).
\end{split}
\end{equation}
where we have also used (\ref{pLessq1}) and (\ref{vy}) that imply
\[\|v(x,\cdot)\|_{L^\infty(\R;\R^m)}\leq 2^\frac{1}{2}\|v(x,\cdot)\|^\frac{1}{2}\|v_y(x,\cdot)\|^\frac{1}{2}< (2\bar{C})^\frac{1}{2}p^\frac{3}{4},\;\;x\in S_{p},\]
and assumed $p$ small. Integrating (\ref{bigger-gamma}) in $(y_j^0,y_j^1)$ yields
\[\begin{split}
&\frac{\vert a\vert}{4}\vert h(\beta_j)-h(\alpha_j)\vert\leq\int_{y_j^0}^{y_j^1}\vert u(\beta_j,y)-u(\alpha_j,y)\vert{d} y
\leq\int_{y_j^0}^{y_j^1}\int_{\alpha_j}^{\beta_j}\vert u_x\vert{d} x{d} y\\
&\leq\frac{1}{2^\frac{1}{2}}\vert h(\beta_j)-h(\alpha_j)\vert^\frac{1}{2}(\beta_j-\alpha_j)^\frac{1}{2}
\Big(\int_{y_j^0}^{y_j^1}\int_{\alpha_j}^{\beta_j}\vert u_x\vert^2{d} x{d} y\Big)^\frac{1}{2}\\
&\leq\vert h(\beta_j)-h(\alpha_j)\vert^\frac{1}{2}(\beta_j-\alpha_j)^\frac{1}{2}\frac{C_0^\frac{1}{2}}{2^\frac{1}{2}}
\end{split}\]
where $C_0$ is the constant in Lemma \ref{ul-bounds}. It follows
\begin{equation}\label{third-s0}
\vert h(\beta_j)-h(\alpha_j)\vert\leq\frac{8 C_0}{\vert a\vert^2}(\beta_j-\alpha_j)
\end{equation}
and in turn
\begin{equation}\label{third-s}
\begin{split}
&\sum_{j\in\tilde{I}_\lambda,\vert h(\beta_j)-h(\alpha_j)\vert>4\bar{y}}\vert h(\beta_j)-h(\alpha_j)\vert\\
&\leq\frac{8 C_0}{\vert a\vert^2}\sum_{j\in\tilde{I}_\lambda,\vert h(\beta_j)-h(\alpha_j)\vert>4\bar{y}}(\beta_j-\alpha_j)\leq\frac{8 C_0}{\vert a\vert^2}\vert\tilde{S}_{p}\vert.
\end{split}
\end{equation}
From (\ref{first-s}), (\ref{secon-s}) and (\ref{third-s}) we conclude that the summation on the right hand side of (\ref{etabar-exp}) is bounded by a constant independent of $L$. It remains to estimate the integral $\int_{S_{p}}\vert h^\prime\vert{d} x$. To do this we use the expression (\ref{hprime}) of $h^\prime$ that implies
\begin{equation}\label{h-abs}
\vert h^\prime(x)\vert\leq\frac{\|v_x(x,\cdot)\|\|v_y(x,\cdot)\|}{\|\bar{u}^\prime+v_y(x,\cdot)\|^2}.
\end{equation}
From (\ref{vy}) we have
\[\|v_y(x,\cdot)\|\leq\frac{1}{2}\|\bar{u}^\prime\|,\;\;x\in S_{p}\]
provided ${p}\leq\frac{\|\bar{u}^\prime\|^2}{4\bar{C}}$. Under this assumption (\ref{h-abs}) and (\ref{vxW}) and (\ref{vyW}) imply
\begin{equation}\label{h-abs1}
\vert h^\prime(x)\vert\leq\frac{2^\frac{1}{2}4\|u_x(x,\cdot)\|^2}{\mu^\frac{1}{2}\|\bar{u}^\prime\|^2}.
\end{equation}
From this estimate we finally obtain
\[
\begin{split}
&\int_{S_{p}}\vert h^\prime\vert{d} x\leq\frac{2^\frac{1}{2}4}{\mu^\frac{1}{2}\|\bar{u}^\prime\|^2}\int_{S_0}\|u_x(x,\cdot)\|^2{d} x)\\
&\leq2\frac{2^\frac{1}{2}4}{\mu^\frac{1}{2}\|\bar{u}^\prime\|^2}C_0.
\end{split}\]
This concludes the proof.
\end{proof}
We are now in the position of refining
 Lemma \ref{first-exp-decay}.
\begin{lemma}\label{exp-deriv}
 There exist constants $k, K>0$ independent of $L>1$ and such that $u=u^L$ satisfies
\begin{equation}\label{exp-decay-3}
\begin{split}
&\vert u(x,y)-a_+\vert\leq K e^{-ky},\quad\text{ for }\;\;y\geq 0,\\
&\vert u(x,y)-a_-\vert\leq K e^{ky},\quad\text{ for }\;\;y\leq 0,
\end{split}
\end{equation}
and
\begin{equation}\label{exp-decay3}
\begin{split}
&\vert (D^\alpha{u})(x,y)\vert\leq K e^{-k\vert y\vert},\;\;\text{ for }\;\; y\in\R, \\
&\|D^\alpha{u}(x,\cdot)\|\leq\frac{K}{\sqrt{2k}},\;\;\text{ for }\;\; x\in[0,L],
\end{split}
\end{equation}
for $\alpha\in\N^2$, $1\leq\vert\alpha\vert\leq 2$.

Moreover, if
\[\min_{\mathrm{p}\in\{-,+\}}\min
_{r\in\R}\|u(x,\cdot)-\bar{u}_{\mathrm{p}}(\cdot-{r})\|\leq q^0\]
and $q^0>0$ is sufficiently small, there is a constant $\bar{C}_1>0$ independent of $L>1$ such that
\begin{equation}\label{v-y-der}
\|D_y^iv(x,\cdot)\|\leq\bar{C}_1,\quad i=0,1,2,
\end{equation}
where $v=v^L$ is defined as in (\ref{vL}).
\end{lemma}
\begin{proof}
From the proof of Proposition \ref{eta-bounded} we have
\begin{equation}
\label{h-Ch}
\vert h(x)\vert\leq C_h,\;\;x\in S_p,
\end{equation}
for some $C_h>0$ independently of $L\geq 1$ and in particular $\vert h(L)\vert=\vert\bar{\eta}\vert\leq C_h$. Since $u=u^L$ satisfies (\ref{smoothness}) and $\bar{u}_\pm$ is bounded in $C^1(\R;\R^m)$, on the basis of Lemma \ref{L2mpliesLinfty}, we can assume that $p>0$ has been chosen so small that
\begin{equation}\label{LinftyB}
\|u(x,\cdot)-\bar{u}_{\mathrm{p}_{x}}(\cdot-h(x))\|_{L^\infty(\R;\R^m)}\leq C p^\frac{2}{3}\leq\frac{r}{4},\;\;x\in S_p,
\end{equation}
for some $r\in(0,\frac{r_0}{2}]$ with $r_0$ as in (\ref{second-derW}).

 Define
\[Y_r:=\{y\geq C_h+y_r: \vert u(x_y,y)-a_+\vert\geq r\;\;\text{for some}\;\;x_y\in(0,L)\},\]
where $y_r>0$ is such that
\begin{equation}\label{yr}
\vert\bar{u}_\pm(y)-a_+\vert\leq\frac{r}{4},\;\;\text{for}\;\;y\geq y_r.
\end{equation}
Note that this and (\ref{h-Ch}) implies
\[\vert\bar{u}_{\mathrm{p}_{x}}(y-h(x))-a_+\vert\leq\frac{r}{4},\;\;\text{for}\;\;y\geq C_h+ y_r.\]
Then (\ref{LinftyB}) yields
\begin{equation}\label{uL-a}
\begin{split}
&\vert u(x,y)-a_+\vert\leq\vert u(x,y)-\bar{u}_{\mathrm{p}_{x}}(y-h(x))\vert
+\vert\bar{u}_{\mathrm{p}_{x}}(y-h(x))-a_+\vert\\
&\leq\frac{r}{2},\;\;x\in S_p,\;y\geq C_h+y_r.
\end{split}
\end{equation}
This inequality shows that $y\in Y_r$ implies that $x_y$ belongs to $\tilde{S}_p$ and therefore to one of the intervals, say $(\alpha,\beta)$, that compose $\tilde{S}_p$. This and (\ref{uL-a}) computed with for $x=\alpha$ yield
\[\vert u(x_y,y)-u(\alpha,y)\vert\geq\vert u(x_y,y)-a_+\vert-\vert u(\alpha,y)-a_+\vert\geq\frac{r}{2},\;\;y\in Y_r.\]
It follows
\[\frac{r}{2}\leq\int_\alpha^{x_y}\vert u_x(x,y)\vert {d} x
\leq\vert\beta-\alpha\vert^\frac{1}{2}\Big(\int_\alpha^\beta\vert u_x(x,y)\vert^2 {d} x\Big)^\frac{1}{2},\]
and in turn Lemma \ref{ul-bounds} implies
\[\vert Y_r\vert\frac{r^2}{4}\leq\vert\tilde{S}_p\vert\int_{\tilde{S}_p}\int_\alpha^\beta\vert u_x(x,y)\vert^2{d} x{d} y\leq C_0\vert\tilde{S}_p\vert,\]
and we have that the measure of $Y_r$ is bounded independently of $L>1$. This shows that there exists an increasing sequence $y_j\rightarrow+\infty$ such that
\[\begin{split}
&y_1\leq C_h+y_r+\vert Y_r\vert,\\
&\vert u(x,y_j)-a_+\vert<r,\;\;x\in[0,L],\;\;j=1,\ldots
\end{split}\]
Therefore, using also (\ref{yr}) that implies
\[\vert u(L,y)-a_+\vert=\vert\bar{u}_+(y-\bar{\eta})-a_+\vert\leq\frac{r}{4},\;\;y\geq C_h+y_r,\]
we can argue as in the proof of Lemma \ref{near-a} and conclude with the help of the Cut-Off Lemma in \cite{af3} that
\[\vert u(x,y)-a_+\vert\leq r,\;\;\text{for}\;\;x\in[0,L],\;\;y\geq y_1.\]
Then the argument in the proof of Lemma \ref{first-exp-decay} yields
\[\vert u(x,y)-a_+\vert\leq r e^{-ky},\;\;y\geq y_1,\]
with $k>0$ independent of $L>1$. The same reasoning show that
\[\vert u(x,y)-a_-\vert\leq r e^{-k\vert y\vert},\;\;y\leq-y_1.\]
This and the fact that $u$ is bounded imply (\ref{exp-decay-3}).
 Once (\ref{exp-decay-3}) is established (\ref{exp-decay3}) follow from elliptic theory.
To prove (\ref{v-y-der}) we observe that (\ref{exp-decay-3}) implies that each solution $\bar{u}\in\{\bar{u}_-,\bar{u}_+\}$, ${h}(x)$ of
\[\|u(x,\cdot)-\bar{u}(\cdot-{h}(x))\|=\min_{\mathrm{p}\in\{-,+\}}\min
_{r\in\R}\|u(x,\cdot)-\bar{u}_{\mathrm{p}}(\cdot-{r})\|,\]
satisfies
\[\vert{h}(x)\vert\leq C_h^\prime,\quad x\in(0,L),\]
for some constant $C_h^\prime>0$ independent of $L>1$. Since $h(x)$ is uniquely determined for $q^0>0$ small, this and (\ref{exp-decay3}) imply that, possibly after adjusting the value of $K$, we have
\[\vert D_y^i({u}(x,y+{h}(x)))\vert\leq K e^{-k\vert y\vert},\;\;x\in[0,L],\; y\in\R,\;i=1,2.\]
Then (\ref{v-y-der}) follows from
\[(D_y^i{v})(x,y)=D_y^i({u}(x,y+{h}(x))-\bar{u}(y)),\;\;i=1,2.
\]
and (\ref{n-bound}). The proof is complete.
\end{proof}

Next, with ${p}\in(0,q^0]$ as before, we focus on the sets
\[\Sigma_\beta:=\{x\in[0,L]:q(x)\leq\beta{p}\}\] for $\beta=1$ and $\beta=\frac{1}{2}$ and show that $\Sigma_{\frac{1}{2}}$ has the simplest possible structure
\begin{proposition}\label{key}
Set $u=u^L$. Then, provided ${p}\in(0,q^0]$ is sufficiently small, there exist $0<l_-<l_+<L$ such that:
\begin{equation}\label{sigma-st}
\Sigma_{\frac{1}{2}}=[0,l_-]\cup[l_+,L]
\end{equation} and
\begin{equation}\label{fixed-diff1}
l_+-l_-\leq C,
\end{equation}
where $C>0$ is a constant independent of $L>1$.
For $x\in[0,l_-]$, the map $\bar{u}\in\{\bar{u}_-,\bar{u}_+\}$ in the representation (\ref{vL}) coincides with $\bar{u}_-$ and for $x\in[l_+,L]$ with $\bar{u}_+$. Moreover the map $x\rightarrow q(x)$ is nondecreasing in $[0,l_-]$ and non increasing in $[l_+,L]$.
\end{proposition}
\begin{proof}

\noindent
1. From (\ref{identity}) and (\ref{v-y-der}) we have, for $x\in\Sigma_1$,
\begin{equation}\label{nu-yy}
\|v_y(x,\cdot)\|^2\leq\|v(x,\cdot)\|\|v_{yy}(x,\cdot)\|\leq\bar{C}_1q(x)\leq\bar{C}_1{p}.
\end{equation}

\noindent
2. For $x\in\Sigma_1$, under the standing assumption that ${p}>0$ is small, we have that $\int_\R\vert u_x\vert^2{d} y$ can be represented as in (\ref{xenergy}) and 1. implies that, for each $x\in\Sigma_1$ the function $f$ introduced in Lemma \ref{represent} is nondecreasing.

\noindent
3. Define
\begin{equation}\label{l-l+}
\begin{split}
& l_-:=\max\{x\in\Sigma_{\frac{1}{2}}:\bar{u}=\bar{u}_-\},\\
& l_+:=\min\{x\in\Sigma_{\frac{1}{2}}:\bar{u}=\bar{u}_+\}
\end{split}
\end{equation}
and observe that the continuity of the map $[0,L]\ni x\rightarrow q(x)\in\R$ and $u(0,\cdot)=\bar{u}_-$, $u(L,\cdot)=\bar{u}_+(\cdot-\bar{\eta})$ imply
\begin{equation}\label{l-l+-value}
\begin{split}
&0<l_\pm<L,\\
&q(l_-)=q(l_+)=\frac{{p}}{2}.
\end{split}
\end{equation}

If $[l_+,L]\not\subset\Sigma_{\frac{1}{2}}$ there exists $x^*\in(l_+,L)$ such that
\[\begin{split}
&\frac{{p}}{2}<q(x^*),\\
&q(x)\leq q(x^*),\quad x\in[l_+,L].
\end{split}
\]
Then there are two possibilities
\begin{description}
\item[a)]$\frac{{p}}{2}<q(x^*)\leq{p}$,
\item[b)]${p}<q(x^*)$.
\end{description}

\noindent
We can immediately exclude case a) by Lemma \ref{NoMax} which, as observed in Remark \ref{TransPoss}, can be applied to the present situation. By arguing as in the final part of the proof of Theorem \ref{main} after Remark \ref{TransPoss}, we can also exclude case b). Indeed if
 $[\tilde{\xi}_1,\tilde{\xi}_2]\subset(\xi_1,\xi_2)$ are defined by
\[\begin{split}
&\tilde{\xi}_1=\min\{x>l_+:q(x)\geq{p}\},\\
&\tilde{\xi}_2=\max\{x:q(x)\geq{p}\},
\end{split}\]
and
\[\begin{split}
&\xi_1=\max\{x<\tilde{\xi}_1:q(x)\leq\frac{{p}}{2}\},\\
&\xi_2=\min\{x>\tilde{\xi}_2:q(x)\leq\frac{{p}}{2}\}.
\end{split}\]
Then, as in (\ref{xtildex-interv}) and (\ref{xtildex-interv1}), we have
\begin{equation}\label{Xtildex-interv}
\begin{split}
&q(\tilde{\xi}_1)=q(\tilde{\xi}_2)={p},\\
&q(\xi_1)=q(\xi_2)=\frac{{p}}{2},\\
&q(x)\in(\frac{{p}}{2},{p}),\;\;x\in(\xi_1,\tilde{\xi}_1)\cup(\tilde{\xi}_2,\xi_2).
\end{split}
\end{equation}
Note that the definition of $\tilde{\xi}_1$ and $\tilde{\xi}_2$ implies $[l_+,\tilde{\xi}_1]\subset\Sigma_1$ and $[\tilde{\xi}_2,L]\subset\Sigma_1$ and therefore we have
\[\bar{u}=\bar{u}_+,\;\;\text{for}\;x\in[l_+,\tilde{\xi}_1]\cup[\tilde{\xi}_2,L].\]
On the basis of these observations we can define a competing map $\tilde{u}$ by setting
\[
\begin{split}
&\tilde{u}(x,\cdot)=u(x,\cdot),\quad x\in[0,\xi_1),\\
&\tilde{u}(x,\cdot)=u(x,\cdot-(\hat{h}(x_2)-h(x_2))),\quad x\in(\xi_2,L].
\end{split}
\]
in $[0,\xi_1)\cup(\xi_2,L]$ and by defining $\tilde{u}$ in the interval $[\xi_1,\xi_2]$  exactly as in the final part of the proof of Theorem \ref{main} (aside from replacing $q^0$ with ${p}$). Then arguing as in that proof we conclude that b) is in contradiction with the minimality of $u$ and $[l_+,L]\subset\Sigma_{\frac{1}{2}}$ is established.

\noindent
Since the proof that $[0,l_-]\subset\Sigma_{\frac{1}{2}}$ is similar we obtain (\ref{sigma-st}). We have $\bar{u}=\bar{u}_-$ in $[0,l_-]$ and $\bar{u}=\bar{u}_+$ in $[l_+,L]$ and therefore $l_-<l_+$. To prove (\ref{fixed-diff1}) we observe that the definition of $\Sigma_{\frac{1}{2}}$ and (\ref{sigma-st}) imply
\[q(x)=\min_{\mathrm{p}\in\{-,+\}}\min_{r\in\R}\|u(x,\cdot)-\bar{u}_{\mathrm{p}}(\cdot-r)\|>\frac{{p}}{2},\quad x\in(l_-,l_+).\]
From this and Lemma \ref{away} we obtain, using also Lemma \ref{ul-bounds}
\[e_{\frac{{p}}{2}}(l_+-l_-)\leq \int_{l_-}^{l_+}(J_\R(u(x,\cdot))-c_0){d} x\leq C_0\]
and (\ref{fixed-diff1}) follows with $C=\frac{C_0}{e_{\frac{{p}}{2}}}$.
The monotonicity of the map $x\rightarrow q(x)$ in the intervals $[0,l_-]$ and $[l_+,L]$ follow from Lemma \ref{NoMax}.
The proof is complete.
\end{proof}

From Proposition \ref{key} we know that the function $x\rightarrow q(x)\leq\frac{{p}}{2}$ is monotone in $[0,l_-]$ and in $[l_+,L]$. Next we show that $q(x)$ converges to $0$ exponentially in $[l_+,L]$ and that a corresponding  statement applies to $[0,l_-]$.
\begin{lemma}\label{exp-decay-in}
It results
\begin{equation}\label{exp-decay-in1}
q(x)
\leq\frac{{p}}{2}e^{-\sqrt{\frac{\mu}{8}}(x-l_+)},\quad\text{ for }\;x\in[l_+,L],
\end{equation}
and
\begin{equation}\label{exp-decay-h}
\vert h^\prime(x)\vert\leq Ce^{-\sqrt{\frac{\mu}{16}}(x-l_+)},\quad\text{ for }\;x\in[l_+,L],
\end{equation}
where $\mu>0$ is the constant in (\ref{W2geqq2}) and $C>0$ is independent of $L>1$.
An analogous statement applies to the interval $[0,l_-]$.
\end{lemma}
\begin{proof}
Proposition \ref{key} implies $q(x)\leq\frac{{p}}{2}$ for $x\in[l_+,L]$ and $q(l_+)=\frac{{p}}{2}$, $q(L)=0$. Therefore we can proceed as in the proof of Lemma \ref{to-baru+} and use (\ref{ElIn0}) and the maximum principle to deduce
\[q(x)^2\leq\varphi(x),\;\;x\in[l_+,L],\]
where
\[\varphi(x)=(\frac{{p}}{2})^2\frac{\sinh{\sqrt{\frac{\mu}{2}}(L-x)}}{\sinh{\sqrt{\frac{\mu}{2}}(L-l_+)}}
\leq(\frac{{p}}{2})^2e^{-\sqrt{\frac{\mu}{2}}(x-l_+)},\quad\text{ for }\;x\in[l_+,L],\]
is the solution of $\varphi^{\prime\prime}=\frac{\mu}{2}\varphi$ with the boundary conditions $\varphi(l_+)=\frac{{p}}{2}$ and $\varphi(L)=0$. This implies (\ref{exp-decay-in1}) and (\ref{exp-decay-h}) follows as in Lemma \ref{to-baru+}. A similar argument applies to the interval $[0,l_-]$.
 The proof is complete.
\end{proof}
\subsection{Conclusion of the proof of Theorem \ref{scha-th}}\label{section34}

We focus on the family of maps $\tilde{u}=\tilde{u}^L$, $L>1$ defined via $u=u^L$, the minimizer in Lemma \ref{ul-exists}, by
\begin{equation}\label{family}
\tilde{u}(x,y):= u(x-l_-,y),\quad (x,y)\in[-l_-,L-l_-]\times\R,\;\;L>1.
\end{equation}
From (\ref{fixed-diff1}) in  Proposition \ref{key} we can assume that, along a subsequence,
\begin{equation}\label{lim-elle}
\lim_{L\rightarrow+\infty}(l_+-l_-)=\ell\leq C.\\
\end{equation}
It follows that, along a further subsequence, at least one between $l_-$ and $L-l_+$ diverges to $+\infty$ as $L\rightarrow+\infty$. Therefore we need to consider two alternatives:
\begin{description}
\item[a)]
\begin{equation}\label{case-1f}
\lim_{L\rightarrow+\infty}l_-=\lim_{L\rightarrow+\infty}L-l_+=+\infty.
\end{equation}
\item[b)] One of the limits in a) is bounded. We will discuss the case (the other case is analogous)
\begin{equation}\label{case-2f}
\lim_{L\rightarrow+\infty}l_-=\ell_-<+\infty,\quad\quad\lim_{L\rightarrow+\infty}L-l_+=+\infty.
\end{equation}
\end{description}
If (\ref{case-1f}) prevails, (\ref{smoothness}) implies that, along a subsequence, we have
\[\lim_{L\rightarrow+\infty}\tilde{u}(x,y)=\mathrm{u}(x,y),\quad (x,y)\in\R^2,\]
where $\mathrm{u}\in C^2(\R^2;\R^m)$ and the convergence is locally in $C^2(\R^2;\R^m)$. It follows that $\mathrm{u}$ is a solution of (\ref{system}) and Lemma \ref{exp-deriv} implies that $\mathrm{u}$ satisfies (\ref{exp-decay-3}).
From Lemma \ref{exp-decay-in} we can also assume that, as $L\rightarrow+\infty$ the functions  $q(\cdot-l_-)$ and $h(\cdot-l_-)$ converge point-wise in $[\ell,+\infty)$ to the functions $q^\mathrm{u}(\cdot)$ and $h^\mathrm{u}(\cdot)$ defined by
\[q^\mathrm{u}(x)=\|\mathrm{u}(x,\cdot)-\bar{u}_+(\cdot-h^\mathrm{u}(x))\|
=\min_{r\in\R}\|\mathrm{u}(x,\cdot)-\bar{u}_+(\cdot-r)\|\]
and moreover that
\[\begin{split}
&q^\mathrm{u}(x)\leq\frac{{p}}{2}e^{-\sqrt{\frac{\mu}{8}}(x-\ell)},\quad x\in[\ell,+\infty),\\
&\vert h_+^\mathrm{u}-h^\mathrm{u}(x)\vert\leq2\frac{C}{\sqrt{\frac{\mu}{16}}}(1-e^{-\frac{\sqrt{\mu}}{2}(x-\ell)}),\quad  x\in[\ell,+\infty)
\end{split}
\]
where $h_+^\mathrm{u}=\lim_{x\rightarrow+\infty}h^\mathrm{u}(x)$.
These estimates proves (\ref{asymptotic-u})$_2$ for $x\rightarrow+\infty$ with $\eta_+=h_+^\mathrm{u}$. A similar reasoning completes the proof of (\ref{asymptotic-u})$_2$ for $x\rightarrow-\infty$. Therefore $\mathrm{u}$ can be identified with the map $u$ in Theorem \ref{scha-th}.

Suppose now that (\ref{case-2f}) holds.
 Proceeding as before we show that along a subsequence, in the limit for $L\rightarrow+\infty$, $\tilde{u}$ converges to a solution $\mathrm{u}:[-\ell,+\infty]\times\R\rightarrow\R^m$ of (\ref{system}) and that $\mathrm{u}$ satisfies (\ref{exp-decay-3}) and (\ref{asymptotic-u})$_2$ for $x\rightarrow+\infty$. On the other hand we have
 \begin{equation}\label{the-good}
 \mathrm{u}(-\ell_-,\cdot)=\bar{u}_+(\cdot-\eta_-)
 \end{equation}
for some $\eta_-\in\R$. Moreover from Lemma \ref{yvar} we have that $\mathrm{u}$ satisfies (\ref{hamilton}) with $\omega=0$ and it follows $\| \mathrm{u}_x(-\ell_-,\cdot)\|=0$ and therefore
 \[ \mathrm{u}_x(-\ell_-,y)=0,\quad y\in\R.\]
 This implies that $\mathrm{u}$ can be extended to $\R^2$ as a $C^1$ map by setting
 \[ \mathrm{u}(x,\cdot)=\bar{u}_-(\cdot-\eta_-),\quad\text{ for }x<-\ell_-.\]
 The map $\mathrm{u}$ extended in this way is a weak solution of (\ref{system}) in $\R^2$ and from the assumption that $W$ is $C^3$ and elliptic theory it follows that $\mathrm{u}$ is a $C^2$ solution of (\ref{system}).  The extended map $\mathrm{u}$ trivially satisfies (\ref{asymptotic-u})$_2$ for $x\rightarrow-\infty$ and therefore we have that also in case b) the map $\mathrm{u}$ can be identified  with the map $u$ in Theorem \ref{scha-th}. The proof is complete.
 \begin{remark}
 Actually the occurrence of case b) can be excluded. Indeed,
  on the basis of the previous discussion, (\ref{case-2f}) implies the existence of two solutions of (\ref{system}) that coincide in an open set, namely $\mathrm{u}$ and the map $\mathrm{v}$ defined by
\[\mathrm{v}(x,y)=\bar{u}_-(y-\eta_-),\quad (x,y)\in\R^2\]
and as observed in \cite{abg} this contradicts the unique continuation theorem in \cite{gl}.
\end{remark}

\nocite{*}
\bibliographystyle{plain}

\begin{thebibliography}{99}



\bibitem{abg}
S.~Alama, L.~Bronsard, and C.~Gui.
\newblock Stationary layered solutions in $\R^2$ for an Allen--Cahn system with multiple well potential.
\newblock {\em Calc.\ Var.} {\bf 5} No.~4 (1997), pp.~359--390.

\bibitem{al}
S.~Alama, and Y.~Y.~Li.
\newblock On "multybamp" bound states for certain semilinear elliptic equations.
\newblock {\em Ind.\ Uni.\ Math.\ Jour.} {\bf 41} (1993), pp.~893--1026.

\bibitem{aac}
\newblock G.~Alberti, L.~Ambrosio, X.~Cabr\'e.
\newblock \emph{On a long-standing conjecture of E. De
Giorgi: simmetry in 3D for general non linearities and a local minimality
property},
\newblock  Acta\ Appl.\ Math. \textbf{65} (2001), pp.~9--33.

\bibitem{a}
N.~D.~Alikakos.
\newblock  Some basic facts on the system $\Delta u-W_u(u)=0$.
\newblock {\em Proc.\ Amer.\ Math.\ Soc.} {\bf 139} No.~1 (2011), pp.~153--162.

\bibitem{abchen}
N.~D.~Alikakos, S.~I.~Betel\'u, and X.~Chen.
\newblock {\em Explicit stationary solutions in multiple well dynamics and non-uniqueness of interfacial energies.}
\newblock  Eur.\ J.\ Appl.\ Math. {\bf 17} (2006), pp.~525--556.



\bibitem{af2}
N.~D.~Alikakos and G.~Fusco.
\newblock On the connection problem for potentials with several global minima.
\newblock {\em Indiana\ Univ.\ Math.\ J.} {\bf 57} No.~4 (2008), pp.~1871--1906.

\bibitem{af3}
N.~D.~Alikakos and G.~Fusco.
\newblock {\em A maximum principle for systems with variational structure and an application to standing waves.}
\newblock Journal of the European Mathematical Society {\bf 17}, No.~7 (2015), pp.~1547--1567.

\bibitem{af}
\newblock N.~D.~Alikakos, G.~Fusco.
\newblock Asymptotic behavior and rigidity results for symmetric
solutions of the
elliptic system $\Delta u = Wu(u)$.
\newblock {\em Ann. \ Sc. \ Norm. \ Super.\ Pisa\ Cl.\ Sci} {\bf 5} Vol.~XV (2016), pp.~809--836.

\bibitem{af4}
N.~D.~Alikakos and G.~Fusco.
\newblock Density estimates for vector minimizers and applications.
\newblock {\em Discr.\ Cont.\ Dynam.\ Syst.} {\bf 35} No.~12 (2015), pp.~5631--5663.


\bibitem{ac}
L.~Ambrosio and X.~Cabr\'e.
\newblock Entire solutions of semilinear elliptic equations in $R^3$ and a conjecture of De Giorgi.
\newblock {\em Amer.\ Math.\ Soc.} {\bf 13} (2000), pp.~ 725--739.

\bibitem{br}
P.~W.~Bates and X.~Ren.
\newblock Transition layers solutions of a higher order equation in an infinite tube.
\newblock {\em Comm.\ Part.\ Diff.\ Equat.} {\bf 21} No.~1-2 (1996), pp.~195--220.

\bibitem{bn}
H.~Berestycki and L.~Nirenberg.
\newblock Travelling fronts in cylinders.
\newblock {\em Ann.\ Inst.\ Henri Poincar\'e, Anal.\ Non Lin\'eaire} {\bf 9} No.~5 (1992), pp.~497--572.

\bibitem{bronsard-gui-schatzman}
L.~Bronsard, C.~Gui, and M.~Schatzman.
\newblock A three-layered minimizer in $\R^2$ for a variational problem with a symmetric three-well potential.
\newblock {\em Comm.\ Pure.\ Appl.\ Math.} {\bf 49} No.~7 (1996), pp.~677--715.

\bibitem{cp}
J.~Carr and B.~Pego.
\newblock Metastable patterns in solutions of $u_t=\epsilon^2 u_{xx}-f(u)$.
\newblock {\em Comm.\ Pure.\ Appl.\ Math.} {\bf 42} No.~5 (1989), pp.~523--576.

\bibitem{cr}
V.~Coti-Zelati and P.~Rabinowitz.
\newblock Homoclinic orbits for second order Hamiltonian systems possessing superquadratic potentials.
\newblock {\em Jour.\ A.M.S.} {\bf 4} (1992), pp.~693--727.

\bibitem{DKW}
M.~Del Pino, M.~Kowalczyk and J.~Wei.
\newblock On De Giorgi's conjecture in dimension $N\geq 9$.
\newblock {\em Annal.\ Math.} {\bf 174} No.~3 (2011), pp.~1485--1569.

\bibitem{fa}
A.~Farina.
\newblock Symmetry for solutions of semilinear elliptic equations in $\R^N$ and
related conjectures.
\newblock {\em Rend.\ Mat.\ Acc.\ Lincei.} {\bf 10} No.~9 (1999), pp.~255--265.



\bibitem{fa1}
A.~Farina.
\newblock On the classification of entire local minimizers of the Ginzburg-Landau equation.
\newblock {\em Contemp.\ Math.} {\bf 595} No.~4 (2013), pp.~231--236.


\bibitem{fsv}
A.~Farina, B.~Sciunzi and E.~Valdinoci.
\newblock Bernstein and De Giorgi type problems: new results via a geometric
approach.
\newblock {\em Ann.\ Sc.\ Norm.\ Super.\ Pisa Cl. Sci.} {\bf 7} No.~4 (2008), pp.~741--791.



\bibitem{fg}
M.~Fazly and N.~Ghoussouby
\newblock De Giorgi type results for elliptic systems.
\newblock {\em Calc.\ Var.\ PDE} {\bf 47 }  (2013), pp.~809--823.

\bibitem{f2}
G.~Fusco.
\newblock On some elementary properties of vector minimizers of the Allen-Cahn energy.
\newblock  {\em \ Comm.\ Pure \ Appl. \ Anal. } {\bf 13} No.~3 (2014),  pp.~1045--1060.


\bibitem{gl}
 N.~Garofalo and F.~H.~Lin.
 \newblock Monotonicity properties of variational integrals, $A_p$ weights an unique continuation.
 \newblock {\em Indiana\ Univ.\ Math.\ J.} {\bf 35} (1986), pp.~245--267.

\bibitem{gg}
 N.~Ghoussoub and C.~Gui.
 \newblock On a conjecture of De Giorgi and some related problems.
 \newblock {\em Math.\ Ann.} {\bf 311} No.~3 (1998), pp.~481--491.

\bibitem{gg1}
 N.~Ghoussoub and C.~Gui.
 \newblock On De Giorgi's conjecture in dimensions 4 and 5.
 \newblock {\em Ann.\ Math.} {\bf 157} No.~1 (2003), pp.~313--334.

\bibitem{gui}
C.~Gui.
\newblock Hamiltonian identities for elliptic differential equations.
\newblock \emph{J.\ Funct.\ Anal.} {\bf 254} No.~4 (2008), pp.~904--933.

\bibitem{he}
D.~Henry
\newblock {\em Geometric theory of semilinear parabolic equations.}
\newblock  Lecture Notes in Mathematics.{\bf 840}\, (1980).

\bibitem{monteil}
 A.~Monteil and F.~Santambrogio.
 \newblock Metric methods for heteroclinic connections.
 \newblock {\em Preprint} (2016).

\bibitem{R}
P.~Rabinowitz.
\newblock Solutions of heteroclinic type for some classes of semilinear elliptic partial differential equations.
\newblock {\em Jour.\ Math.\ Sci.\ Uni.\ Tokyo} {\bf 2} (1994), pp.~525--550.


\bibitem{r}
W.~Rudin
\newblock {\em Functional Analysis.}
\newblock  McGraw-Hill Series in Higher Math.\, (1973).


\bibitem{s}
O.~Savin.
\newblock Regularity of
 level sets in phase transitions.
\newblock {\em Ann.\ Math.} (2) {\bf 169} No.~1 (2009), pp.~41--78.


\bibitem{scha}
M.~Schatzman.
\newblock Asymmetric heteroclinic double layers.
\newblock \emph{ESAIM Control\ Optim.\ Calc.\ Var.} {\bf 8} No.~ (2002) (A tribute to J. L. Lions), pp.~965--1005. (electronic)
http://dx.doi.org/10.1051/cocv:2002039. MR 1932983 (2003i:35074).

\bibitem{SoaveTerracini}
N.~Soave and S.~Terracini.
\newblock Liouville theorems and 1-dimensional symmetry for solutions of an elliptic system modeling phase separation.
\newblock {\em  Adv.\ Math.} {\bf 279} (2015), pp.~29--66.

\bibitem{sourdis}
 C.~Sourdis.
 \newblock The heteroclinic connection problem for general double-well potentials.
 \newblock {\em Preprint} http://arxiv.org/abs/1311.2856

\bibitem{ZS}
 A.~Zuniga and P.~Sternberg.
 \newblock On the heteroclinic connection problem for multi-well gradient systems.
 \newblock {\em Preprint} (2016).
 \end{thebibliography}

\end{document}